\documentclass{amsproc}

\usepackage{verbatim}
\usepackage{xcolor}
\usepackage{tikz}
\usepackage{tikz-cd}
\usepackage{arydshln}
\usepackage{caption}
\usepackage{subcaption}
\usepackage{graphicx}
\tikzset{
	symbol/.style={
		draw=none,
		every to/.append style={
			edge node={node [sloped, allow upside down, auto=false]{$#1$}}}
	}
}

\usepackage{amsmath}
\usepackage{bbm}
\usepackage{hyperref}

\newtheorem{theorem}{Theorem}[section]
\newtheorem{proposition}[theorem]{Proposition}
\newtheorem{lemma}[theorem]{Lemma}
\newtheorem{corollary}[theorem]{Corollary}
\newcommand{\arcsinh}{\operatorname{arcsinh}}
\theoremstyle{definition}

\theoremstyle{remark}
\newtheorem{remark}[theorem]{Remark}
\newtheorem{question}[theorem]{Question}

\numberwithin{equation}{section}

\begin{document}

\title[Square-integrability of the Mirzakhani function]{Square-integrability of the Mirzakhani function and statistics of simple closed geodesics on hyperbolic surfaces}


%    author one information
%\author{Francisco Arana--Herrera}
%\address{Department of Mathematics, Stanford University, 450 Serra Mall
%	Building 380, Stanford, CA 94305-2125, USA}
%\curraddr{}
%\email{farana@stanford.edu}
%\thanks{}

\author{Francisco Arana--Herrera}
\author{Jayadev S.~Athreya}

\email{farana@stanford.edu}
\email{jathreya@uw.edu}

\address{Department of Mathematics, Stanford University, 450 Serra Mall
	Building 380, Stanford, CA 94305-2125, USA}
\address{Department of Mathematics, University of Washington, Padelford Hall, Seattle, WA 98195-4350, USA}

\thanks{J.S.A. partially supported by NSF CAREER grant DMS 1559860}

%    author two information
%\author{}
%\address{}
%\curraddr{}
%\email{}
%\thanks{}

%\subjclass[2000]{Primary }
%    The 2010 edition of the Mathematics Subject Classification is
%    now available.  If you are citing a classification from the
%    new scheme, use the following input coding instead.
%\subjclass[2010]{Primary }

\date{}

\begin{abstract}
Given integers $g,n \geq 0$ satisfying $2-2g-n < 0$, let $\mathcal{M}_{g,n}$ be the moduli space of connected, oriented, complete, finite area hyperbolic surfaces of genus $g$ with $n$ cusps. We study the global behavior of the Mirzakhani function $B \colon \mathcal{M}_{g,n} \to \mathbf{R}_{\geq 0}$ which assigns to $X \in \mathcal{M}_{g,n}$ the Thurston measure of the set of measured geodesic laminations on $X$ of hyperbolic length $\leq 1$. We improve bounds of Mirzakhani describing the behavior of this function near the cusp of $\mathcal{M}_{g,n}$ and deduce that $B$ is square-integrable with respect to the Weil-Petersson volume form. We relate this knowledge of $B$ to statistics of counting problems for simple closed hyperbolic geodesics.\\
\end{abstract}

\maketitle


\thispagestyle{empty}

\tableofcontents

\section{Introduction}

$ $

In \cite{Mir08b}, Mirzakhani gave a precise description of the growth of the number of simple closed geodesics of length $\leq L$ of a fixed topological type on an arbitrary connected, orientable, complete, finite area hyperbolic surface as $L \to \infty$. Fix integers $g,n \geq 0$ satisfying $2-2g-n < 0$. Let $\mathcal{T}_{g,n}$ and $\mathcal{M}_{g,n}$ be the Teichmüller and moduli spaces of connected, oriented, complete, finite area hyperbolic surfaces of genus $g$ with $n$ cusps. Fix a connected, oriented surface $S_{g,n}$ of genus $g$ with $n$ punctures and let $\text{Mod}_{g,n}$ be its mapping class group. Given a rational multi-curve $\gamma$ on $S_{g,n}$ and a hyperbolic surface $X \in \mathcal{M}_{g,n}$, Mirzakhani considered for every $L > 0$ the counting function
\begin{equation}
\label{eq:count}
s(X,\gamma,L) := \# \{\alpha \in \text{Mod}_{g,n} \cdot \gamma \ | \ \ell_\gamma(X) \leq L \},
\end{equation}
where $\ell_\gamma(X) >0 $ denotes the hyperbolic length of the unique geodesic representative of $\gamma$ in $X$. The following theorem, corresponding to Theorem 1.1 in \cite{Mir08b}, shows that $s(X,\gamma,L)$ behaves asymptotically like a polynomial of degree $6g-6+2n$ when $L \to \infty$.\\

\begin{theorem}\cite[Theorem 1.1]{Mir08b}
	\label{theo:mir_count}
	For any $X \in \mathcal{M}_{g,n}$ and any rational multi-curve $\gamma$ on $S_{g,n}$,
	\[
	\lim_{L \to \infty} \frac{s(X,\gamma,L)}{L^{6g-6+2n}} = n_\gamma(X),
	\]
	where $n_{\gamma}(X) \colon \mathcal{M}_{g,n} \to \mathbf{R}_{>0}$ is a continuous proper function.\\
\end{theorem}

Mirzakhani's understanding of the asymptotics of $s(X,\gamma,L)$ go even deeper: she provides an explicit description of the dependency of the leading coefficient $n_{\gamma}(X)$ on the rational multi-curve $\gamma$ and the hyperbolic surface $X$. More precisely, let $\mathcal{ML}_{g,n}$ be the space of measured geodesic laminations on $S_{g,n}$ and $\mu_{\text{Thu}}$ be the Thurston measure on $\mathcal{ML}_{g,n}$. Consider the function $B \colon \mathcal{M}_{g,n} \to \mathbf{R}_{>0}$ given by
\[
B(X) := \mu_{\text{Thu}}(\{\lambda \in \mathcal{ML}_{g,n} \ | \ \ell_{\lambda}(X) \leq 1 \}),
\]
where $\ell_{\lambda}(X) > 0$ denotes the hyperbolic length of the measured geodesic lamination $\lambda$ on $X$. The following proposition corresponds to Proposition 3.2 and Theorem 3.3. in \cite{Mir08b}.\\

\begin{proposition}\cite[Proposition 3.2 and Theorem 3.3]{Mir08b}
	\label{prop:Borig}
	The function $B \colon \mathcal{M}_{g,n} \to \mathbf{R}_{>0}$ is continuous, proper, and integrable with respect to the Weil-Petersson volume form on $\mathcal{M}_{g,n}$.\\
\end{proposition}

Let $\widehat{\mu}_{wp}$ be the measure induced by the Weil-Petersson volume form on $\mathcal{M}_{g,n}$. We consider the normalization constant
\[
b_{g,n} := \int_{\mathcal{M}_{g,n}} B(X) \ d\widehat{\mu}_{\text{wp}}(X) < \infty.
\]
$ $

The last result needed to describe the leading coefficient $n_\gamma(X)$ is the following proposition, corresponding to Corollary 5.2 in \cite{Mir08b}.\\

\begin{proposition}\cite[Corollary 5.2]{Mir08b}
	\label{prop:mir_freq}
	For any rational multi-curve $\gamma$ on $S_{g,n}$, the integral
	\[
	P(L,\gamma) := \int_{\mathcal{M}_{g,n}} s(X,\gamma,L) \ d\widehat{\mu}_{\text{wp}}(X)
	\]
	is a polynomial of degree $6g-6+2n$ in $L > 0$ with non-negative coefficients and whose leading coefficient
	\[
	c(\gamma) := \lim_{L \to \infty} \frac{P(L,\gamma)}{L^{6g-6+2n}}
	\]
	is a positive rational number.\\
\end{proposition}

The following theorem, corresponding to Theorem 1.2 in \cite{Mir08b}, describes the dependency of the leading coefficient $n_\gamma(X)$ on the rational multi-curve $\gamma$ and the hyperbolic surfaces $X$.\\

\begin{theorem}\cite[Theorem 1.2]{Mir08b}
	\label{theo:lead}
	For every rational multi-curve $\gamma$ on $S_{g,n}$ and every $X \in \mathcal{M}_{g,n}$,
	\[
	n_\gamma(X) = \frac{c(\gamma) \cdot B(X)}{b_{g,n}}.
	\]
\end{theorem}
$ $

The constant $c(\gamma) \in \mathbf{Q}_{>0}$ is usually referred to as the frequency of $\gamma$, or more precisely, as the frequency of rational multi-curves of the same topological type as $\gamma$. Notice that 
\[
\int_{\mathcal{M}_{g,n}} \eta_{\gamma}(X) \ d\widehat{\mu}_{\text{wp}}(X) = c(\gamma).
\]
We will refer to the function $B \colon \mathcal{M}_{g,n} \to \mathbf{R}$ as the \textit{Mirzakhani function}.\\

Recently, see \cite{EMM19}, Eskin, Mirzakhani, and Mohammadi improved Theorem \ref{theo:mir_count} by obtaining a power saving error term for the asymptotics. Their methods are very different from Mirzakhani's original work and rely on the exponential mixing rate of the Teichmüller geodesic flow.\\

The Mirzakhani function plays a crucial role in the study of the moduli space $\mathcal{M}_{g,n}$ from the perspective of hyperbolic geometry:\\
\begin{enumerate}
	\item As seen in Theorem \ref{theo:lead}, $B(X)$ describes the dependency on the hyperbolic metric $X$ of the leading coefficient of the asymptotics of counting problems for simple closed hyperbolic geodesics.\\
	
	\item The bundle $\pi \colon P^1\mathcal{M}_{g,n} \to \mathcal{M}_{g,n}$ of length $1$ measured geodesic laminations over moduli space carries a unique (up to scaling) Lebesgue class measure $\nu$ invariant and ergodic with respect to the earthquake flow. We refer to $\nu$ as the \textit{Mirzakhani measure}. The pushforward $\pi_* \nu$ is absolutely continuous with respect to the Weil-Petersson measure $\widehat{\mu}_{\text{wp}}$ and its density is precisely given by the Mirzakhani function $B \colon \mathcal{M}_{g,n} \to \mathbf{R}_{>0}$, i.e., $\pi_* \nu = B(X) \ d \widehat{\mu}_{\text{wp}}(X)$. See \cite{Mir08a} for details.\\
	
	\item The function $B \colon \mathcal{M}_{g,n} \to \mathbf{R}_{>0}$ is the asymptotic distribution with respect to the Weil-Petersson measure $\widehat{\mu}_{\text{wp}}$ of random hyperbolic surfaces constructed in the following way. Fix a pair of pants decomposition $\mathcal{P} := \{\gamma_1,\dots,\gamma_{3g-3+n}\}$ of $S_{g,n}$. Let $L > 0$ be arbitrary. Pick uniformly at random a point from the simplex of vectors $(\ell_i)_{i=1}^{3g-3+n} \in \mathbf{R}^{3g-3+n}$ with positive entries satifying $\ell_1 + \cdots + \ell_{3g-3+n} \leq L$. Consider the pairs of pants in the decomposition induced by $\mathcal{P}$ on $S_{g,n}$ as hyperbolic pairs of pants with cuff lengths given by $\ell(\gamma_i) = \ell_i$ for all $i = 1,\dots,3g-3+n$. Choose twist parameters $0 \leq \tau_i < \ell_i$ uniformly at random for every $i=1,\dots,3g-3+n$ and glue the hyperbolic pairs of pants according to these twist parameters to get a random hyperbolic surface on $\mathcal{M}_{g,n}$. Let $\widehat{\mu}^L_{\mathcal{P},*}$ be the probability measure on $\mathcal{M}_{g,n}$ describing such random hyperbolic surface. Then, as $L \to \infty$, 
	\[
	\widehat{\mu}^L_{\mathcal{P},*} \to \frac{B(X) \ d\widehat{\mu}_{\text{wp}}}{b_{g,n}}.
	\]
	Several other similar constructions, for instance, choosing lengths uniformly at random from the codimension $1$ simplex  of vectors $(\ell_i)_{i=1}^{3g-3+n} \in \mathbf{R}^{3g-3+n}$ with positive entries satisfying $\ell_1,\dots,\ell_{3g-3+n} = L$, or considering simple closed multi-curves more general than a pair of pants decomposition, exhibit the same behavior. These results are all consequences of the ergodicity of the earthquake flow with respect to the Mirzakhani measure $\nu$ on $P^1\mathcal{M}_{g,n}$. See \cite{Mir07b} and \cite{AH19b} for details.\\
\end{enumerate}

It is also interesting to note that $b_{g,n}$, the integral of $B$ with respect to the Weil-Petersson measure $\widehat{\mu}_{\text{wp}}$ on $\mathcal{M}_{g,n}$, corresponds to the Masur-Veech measure of the principal stratum of $Q\mathcal{M}_{g,n}$, the moduli space of connected, integrable, meromorphic quadratic differentials of genus $g$ with $n$ marked points. See Theorem 1.4 in \cite{Mir08a} or Corollary 1.4 in \cite{AH19a} for details.\\

Aside from Mirzakhani's original understanding of the function $B \colon \mathcal{M}_{g,n} \to \mathbf{R}$, roughly described in Proposition \ref{prop:Borig}, not much is known about the behavior of the Mirzakhani function, both locally and globally. The main purpose of this paper is to strengthen our understanding of the global behavior of the Mirzakhani function and to describe in more depth its connections to the statistics of counting problems for simple closed hyperbolic geodesics.\\

\textit{Main results.} The goal of the first part of this paper is to better understand the global behavior of the Mirzakhani function. Consider the function $R \colon \mathbf{R}_+ \to \mathbf{R}_+$ given by
\begin{equation}
\label{eq:fnR}
R(x) := \frac{1}{x \cdot | \text{log}(x)|}.
\end{equation}
The following bounds coarsely describe the behavior of the Mirzakhani function near the cusp in terms of the lengths of short simple closed hyperbolic geodesics.\\

\begin{theorem}
	\label{theo:main_res_1}
	For all sufficiently small $\epsilon >0$ there are constants $C_1, C_2 > 0$ such that for all $X \in \mathcal{M}_{g,n}$,
	\[
	C_1 \cdot \prod_{\gamma \colon \ell_{\gamma}(X) \leq \epsilon} R(\ell_\gamma(X)) \leq	B(X) \leq C_2 \cdot \prod_{\gamma \colon \ell_{\gamma}(X) \leq \epsilon} R(\ell_\gamma(X)) ,
	\]
	where the products range over all simple closed geodesics $\gamma$ in $X$ of length $\leq \epsilon$.
\end{theorem}
$ $

\begin{remark}
	In Theorem \ref{theo:main_res_1} and Proposition \ref{prop:main_res_3} below, the values of $\epsilon > 0$ considered are small enough so that on any hyperbolic surface no two simple closed geodesics of length $\leq \epsilon$ intersect. In particular, the products involved in the statements of these results range over a finite collection of pairwise disjoint simple closed geodesics.\\
\end{remark}

The lower bound in Theorem \ref{theo:main_res_1} and a weaker upper bound are proved in Proposition 3.6 of \cite{Mir08b}. Our proof of Theorem \ref{theo:main_res_1} follows the same ideas as the proof of Proposition 3.6 in \cite{Mir08b} but using more precise estimates.\\

As a direct consequence of Theorem \ref{theo:main_res_1} we obtain the following result.\\

\begin{theorem}
	\label{theo:main_res_2}
	The Mirzakhani function $B \colon \mathcal{M}_{g,n} \to \mathbf{R}_{>0}$ is square integrable with respect to the Weil-Petersson measure $\widehat{\mu}_{\text{wp}}$ on $\mathcal{M}_{g,n}$, i.e.,
	\[
	a_{g,n} := \int_{\mathcal{M}_{g,n}} B(X)^2 \ d\widehat{\mu}_{\text{wp}}(X) < \infty
	\]
\end{theorem}
$ $

Let $m_{g,n} := \widehat{\mu}_{\text{wp}}(\mathcal{M}_{g,n})$. Theorem \ref{theo:main_res_2} states that the random variable $$B \colon \mathcal{M}_{g,n} \allowbreak \to \mathbf{R}_{>0}$$ defined on the probability space $(\mathcal{M}_{g,n}, \widehat{\mu}_{\text{wp}}/m_{g,n})$ has finite second moment. In particular, one can consider its variance
\[
\mathbf{Var}(B(X)) := \mathbf{E}(B(X)^2) - \mathbf{E}(B(X))^2 = \frac{a_{g,n}}{m_{g,n}} - \frac{b_{g,n}^2}{m_{g,n}^2} < \infty.
\]
$ $

The focus of the second part of this paper is to relate the newly acquired knowledge of the global behavior of the Mirzakhani function to the statistics of counting problems for simple closed hyperbolic geodesics. We are interested in studying the asymptotic behavior of the covariances of the counting functions $s(X,\gamma,L)$ defined on $\mathcal{M}_{g,n}$ for different integral multi-curves $\gamma$ on $S_{g,n}$ as $L \to \infty$. To this end, and inspired by the definition of the frequencies $c(\gamma)$ in Proposition \ref{prop:mir_freq}, for every pair of integral multi-curves $\gamma_1,\gamma_2$ on $S_{g,n}$, we define their \textit{joint frequency} $c(\gamma_1,\gamma_2)$ to be the limit
\begin{equation}
\label{eq:joint_freq}
c(\gamma_1,\gamma_2) := \lim_{L \to \infty} \frac{1}{L^{12g-12+4n}} \int_{\mathcal{M}_{g,n}} s(X,\gamma_1,L) \cdot s(X,\gamma_2,L) \ d\widehat{\mu}_{\text{wp}}(X).
\end{equation}
$ $

The main result of the second part of this paper is the following theorem, which establishes the existence of the joint frequencies $c(\gamma_1,\gamma_2)$ and provides a formula relating them to the frequencies $c(\gamma_1)$ and $c(\gamma_2)$ through the constants $b_{g,n}$ and $a_{g,n}$.\\

\begin{theorem}
	\label{theo:main_res_4}
	For any pair of integral multi-curves $\gamma_1, \gamma_2$ on $S_{g,n}$, the limit in the definition (\ref{eq:joint_freq}) of $c(\gamma_1,\gamma_2)$ exists and moreover,
	\[
	c(\gamma_1,\gamma_2) = \frac{a_{g,n}}{b_{g,n}^2} \cdot c(\gamma_1) \cdot c(\gamma_2).
	\]
\end{theorem}
$ $

A key ingredient in the proof of Theorem \ref{theo:main_res_4} is the following bound, interesting in its own right; this bound is obtained through similar methods as the upper bound in Theorem \ref{theo:main_res_1}. \\

\begin{proposition}
	\label{prop:main_res_3}
	For all sufficiently small $\epsilon >0$, there exist constants $C > 0$ and $L_0 > 0$ such that for all $L \geq L_0$, all $X \in \mathcal{M}_{g,n}$, and all integral multi-curves $\eta$ on $S_{g,n}$,
	\[
	\frac{s(X,\eta,L)}{L^{6g-6+2n}} \leq C \cdot \prod_{\gamma \colon \ell_{\gamma}(X) \leq \epsilon} R(\ell_\gamma(X)).
	\]
	where the product ranges over all simple closed geodesics $\gamma$ in $X$ of length $\leq \epsilon$.
\end{proposition}
$ $

It turns out that the constants $b_{g,n}$ and  $a_{g,n}$ can be recovered from the frequencies $c(\gamma)$ and the joint frequencies $c(\gamma_1,\gamma_2)$, respectively. More precisely, let $\mathcal{ML}_{g,n}(\mathbf{Z})$ be the set of all integral multi-curves on $S_{g,n}$. Following the ideas introduced in the proof of Theorem 5.3 in \cite{Mir08b}, we obtain the following formulas:\\

\begin{theorem}
	\label{theo:main_res_5}
	For any integers $g,n \geq 0$ such that $2 -2g -n < 0$,
	\begin{align*}
	b_{g,n} &= \sum_{\gamma \in \mathcal{ML}_{g,n}(\mathbf{Z})/\text{Mod}_{g,n}} c(\gamma), \\
	a_{g,n} &= \sum_{\gamma_1,\gamma_2 \in \mathcal{ML}_{g,n}(\mathbf{Z})/\text{Mod}_{g,n}} c(\gamma_1,\gamma_2).
	\end{align*}
\end{theorem}
$ $

Consider the probability space $(\mathcal{M}_{g,n}, \widehat{\mu}_{\text{wp}}/m_{g,n})$. According to Proposition \ref{prop:mir_freq}, the expected value of the counting function $s(X,\gamma,L)$
\[
\mathbf{E}(\gamma,L) := \mathbf{E}(s(\gamma,X,L)) = \frac{1}{m_{g,n}} \cdot \int_{\mathcal{M}_{g,n}} s(X,\gamma,L) \ d\widehat{\mu}_{\text{wp}}(X)
\]
is a polynomial of degree $6g-6+2n$ in the variable $L$ with leading coefficient 
\[
\mathbf{E}(\gamma) := \lim_{L \to \infty} \frac{\mathbf{E}(\gamma,L)}{L^{6g-6+2n}} = \frac{c(\gamma)}{m_{g,n}}.
\]
According to Theorem \ref{theo:main_res_4}, the covariance of the counting functions $s(X,\gamma_1,L)$ and $s(X,\gamma_2,L)$ 
\begin{align*}
\mathbf{Cov}(\gamma_1,\gamma_2,L) &:= \mathbf{Cov}(s(X,\gamma_1,L),s(X,\gamma_2,L))\\
&= \frac{1}{m_{g,n}}  \cdot\int_{\mathcal{M}_{g,n}} s(X,\gamma_1,L) \cdot s(X,\gamma_2,L) \ d\widehat{\mu}_{\text{wp}}(X)\\
&\phantom{=}-\frac{1}{m_{g,n}^2} \cdot \left( \int_{\mathcal{M}_{g,n}} s(X,\gamma_1,L)  \ d\widehat{\mu}_{\text{wp}}(X)\right) \left( \int_{\mathcal{M}_{g,n}}  s(X,\gamma_2,L) \ d\widehat{\mu}_{\text{wp}}(X)\right)
\end{align*}
behaves asymptotically, as $L \to \infty$, like a polynomial of degree $12g-12+4n$ in $L$ with leading coefficient
\[
\mathbf{Cov}(\gamma_1,\gamma_2) := \lim_{L \to \infty} \frac{\mathbf{Cov}(\gamma_1,\gamma_2,L)}{L^{12g-12+4n}} = \frac{c(\gamma_1,\gamma_2)}{m_{g,n}} - \frac{c(\gamma_1) \cdot c(\gamma_2)}{m_{g,n}^2}.
\]
Theorem \ref{theo:main_res_4} establishes the following relation
\[
\mathbf{Cov}(\gamma_1,\gamma_2) = \frac{\mathbf{Var}(B(X))}{\mathbf{E}(B(X))^2} \cdot \mathbf{E}(\gamma_1) \cdot \mathbf{E}(\gamma_2).
\]
Theorem \ref{theo:main_res_5} shows
\begin{align*}
\mathbf{E}(B(X)) &= \sum_{\gamma \in \mathcal{ML}_{g,n}(\mathbf{Z})/\text{Mod}_{g,n}} \mathbf{E}(\gamma),\\
\mathbf{Var}(B(X)) &= \sum_{\gamma_1,\gamma_2 \in \mathcal{ML}_{g,n}(\mathbf{Z})/\text{Mod}_{g,n}} \mathbf{Cov}(\gamma_1,\gamma_2).
\end{align*}
$ $

\textit{Organization of the paper.} In \S\ref{background} we present the background material necessary to understand the proofs of the main results. In Section 3 we discuss the global behavior of the Mirzakhani function and present the proofs of Theorems \ref{theo:main_res_1} and \ref{theo:main_res_2}. In Section 4 we discuss the connections between the newly acquired knowledge of the global behavior of the Mirzakahani function and the statistics of counting problems for simple closed hyperbolic geodesics; we prove Proposition \ref{prop:main_res_3}, and Theorems \ref{theo:main_res_4} and \ref{theo:main_res_5}. In Section 5 we present a series of open questions that arise naturally from the work in this paper.\\

\textit{Acknowledgments.} The authors are very grateful to Scott Wolpert for his enlightening conversations, his careful reading of an earlier version of this paper, and his detailed comments and suggestions that greatly helped improve the exposition. The first author would also like to thank Alex Wright and Steven Kerckhoff for their invaluable advice, patience, and encouragement.\\

\section{Background material}\label{background}

$ $

\textit{Notation.} Let $g,n \geq 0$ be integers such that $2 - 2g - n < 0$. For the rest of this paper, $S_{g,n}$ will denote a connected, oriented, smooth surface of genus $g$ with $n$ punctures (and negative Euler characteristic). For $g \geq 0$ we will also use the notation $S_{g} := S_{g,0}$. Unless otherwise specified, when applied to simple closed curves or measured geodesic laminations, the word length will always mean hyperbolic length.\\

\textit{Teichm\"uller and moduli spaces of hyperbolic surfaces.} The Teichm\"uller space of $S_{g,n}$, denoted $\mathcal{T}_{g,n}$, is the space of all marked oriented, complete, finite area hyperbolic structures on $S_{g,n}$ up to isotopy. More precisely, $\mathcal{T}_{g,n}$ is the space of pairs $(X,\phi)$, where $X$ is an oriented, complete, finite area hyperbolic surface and $\phi \colon S_{g,n} \to X$ is an orientation-preserving diffeomorphism, modulo the equivalence relation $(X,\phi_1) \sim (X,\phi_2)$ if and only if there exists an orientation-preserving isometry $I \colon X_1 \to X_2$ isotopic to $\phi_2 \circ \phi_1^{-1}$. \\

Given a marked hyperbolic surface $(X,\phi) \in \mathcal{T}_{g,n}$ and a simple closed curve $\gamma$ on $S_{g,n}$, we will denote by $\ell_\gamma(X) > 0$ the hyperbolic length of the unique geodesic representative of $\phi(\gamma)$ on $X$; we usually omit the markings in the notation and simply say that this is the length of the geodesic representative of $\gamma$ on $X$. Given a pair of pants decomposition $\mathcal{P} := \{\gamma_1,\dots,\gamma_{3g-3+n}\}$ of $S_{g,n}$, the length functions $\ell_i := \ell_{\gamma_i} \colon \mathcal{T}_{g,n} \to \mathbf{R}_{>0}$ can be complemented with twist parameters $\tau_i \colon \mathcal{T}_{g,n} \to \mathbf{R}$ to obtain a set of coordinates $(\ell_i,\tau_i)_{i=1}^{3g-3+n} \in (\mathbf{R}_{>0} \times \mathbf{R})^{3g-3+n}$ for $\mathcal{T}_{g,n}$. Any such set of coordinates is called a set of Fenchel-Nielsen coordinates of $\mathcal{T}_{g,n}$ adapted to $\mathcal{P}$. See \cite[\S 10.6]{FM11} for more details.\\

We denote the mapping class group of $S_{g,n}$ by $\text{Mod}_{g,n}$. The mapping class group of $S_{g,n}$ acts properly discontinuously on $\mathcal{T}_{g,n}$ by change of marking. The quotient $\mathcal{M}_{g,n} := \mathcal{T}_{g,n} / \text{Mod}_{g,n}$ is the moduli space of oriented, complete, finite area hyperbolic structures on $S_{g,n}$. \\

\textit{The Weil-Peterson volume form.} The Teichm\"ueller space $\mathcal{T}_{g,n}$ can be endowed with a $3g-3+n$ dimensional complex structure. This complex structure admits a natural K\"ahler Hermitian structure. The associated symplectic form $\omega_{\text{wp}}$ is called the Weil-Petersson symplectic form. The Weil-Petersson volume form is the top exterior power $v_{\text{wp}} := \frac{1}{(3g-3+n)!}\bigwedge^{3g-3+n} \omega_{\text{wp}}$. The Weil-Petersson measure on $\mathcal{T}_{g,n}$ is the measure $\mu_{\text{wp}}$ induced by $v_{\text{wp}}$. The Weil-Petersson measure $\widehat{\mu}_{\text{wp}}$ on $\mathcal{M}_{g,n}$ is the local pushforward of $\mu_{\text{wp}}$ under the quotient map $\mathcal{T}_{g,n} \to \mathcal{M}_{g,n}$. See \cite{Hub16} for more details.\\

In \cite{Wol85}, Wolpert obtained the following expression for $\omega_{\text{wp}}$, valid for any set of Fenchel-Nielsen coordinates $(\ell_i,\tau_i)_{i=1}^{3g-3+n} \in (\mathbf{R}_{>0} \times \mathbf{R})^{3g-3+n}$ of $\mathcal{T}_{g,n}$, commonly known as \textit{Wolpert's magic formula}:
\[
\omega_{\text{wp}} = \sum_{i=1}^{3g-3+n} d \ell_i \wedge d\tau_i.
\] 
In particular, the Weil-Petersson volume form $v_{\text{wp}}$ can be expressed in Fenchel-Nielsen coordinates as follows:
\[
v_{\text{wp}} = \prod_{i=1}^{3g-3+n} d \ell_i \wedge d \tau_i.
\]
$ $

\textit{The collar lemma.} The following theorem, commonly known as the \textit{collar lemma}, shows that short geodesics on hyperbolic surfaces admit wide embedded collar neighborhoods; see \cite[\S 13.5]{FM11} for details.\\

\begin{theorem}
	\label{theo:collar}
	Let $\gamma$ be a simple closed geodesic on a hyperbolic surface $X$. Then $N_{\gamma} := \{x \in X  \colon  d(x,\gamma) \leq w(\ell_\gamma(X))\}$ is an embedded annulus, where $w \colon \mathbf{R}_{>0} \to \mathbf{R}_{>0}$ is the function given by
	\[
	w(x) := \arcsinh\left( \frac{1}{\sinh \left(\frac x 2 \right)}\right).
	\]
\end{theorem}
$ $

One can check that
\[
\lim_{x \to 0^+} \frac{w(x)}{|\log(x)|} = 1.
\]
In particular, 
\[
\lim_{x \to 0^+} w(x) = + \infty,
\]
i.e., short simple closed geodesics develop wide collars. As a consequence, one can find a universal constant $\epsilon > 0$ such that on any hyperbolic surface $X$ no two simple closed geodesics of length $\leq \epsilon$ intersect.\\

\textit{The Bers constant.} In \cite{Ber85}, Bers proved that every connected, orientable, closed hyperbolic suface $X$ of genus $g \geq 2$ admits a pair of pants decomposition $\mathcal{P} := \{\gamma_1,\dots,\gamma_{3g-3}\}$ satisfying
\[
\ell_{\gamma_i}(X) \leq L_g, \ \forall i=1,\dots,3g-3,
\]
where $L_g>0$ is a constant depending only on $g$. The best possible constant with such property is commonly know as the \textit{Bers' constant} of $S_g$. The following non-optimal version of Bers' theorem allows punctures and will be enough for our purposes; see \cite[\S 5]{Bus92} and \cite[\S 12.4.2]{FM11} for more details.\\

\begin{theorem}
	\label{theo:bers}
	Let $X \in \mathcal{T}_{g,n}$ be a marked hyperbolic structure and $\gamma_1,\dots, \gamma_k$ be pairwise disjoint, pairwise non-isotopic simple closed curves on $S_{g,n}$ satisfying
	\[
	\ell_{\gamma_i}(X) \leq 1, \ \forall i=1,\dots,k.
	\]
	Such a collection of curves can be completed to a pair of pants decomposition 
	\[
	\mathcal{P} := \{\gamma_1,\dots,\gamma_k,\gamma_{k+1}, \dots, \gamma_{3g-3+n}\}
	\]
	of $S_{g,n}$ satisfying
	\[
	\ell_{\gamma_i}(X) \leq L_{g,n}, \ \forall i=1,\dots,3g-3+n,
	\]
	where $L_{g,n} > 1$ is a constant depending only on $g$ and $n$.\\
\end{theorem}
$ $

\textit{Measured geodesic laminations and singular measured foliations.} A geodesic lamination $\lambda$ on a complete, finite area hyperbolic surface $X$ diffeomorphic to $S_{g,n}$ is a set of disjoint simple, complete geodesics whose union is a compact subset of $X$. A measured geodesic lamination is a geodesic lamination carrying an invariant transverse measure fully supported on the lamination. We can understand measured geodesic laminations by lifting them to a universal cover $\mathbf{H}^2 \to X$. A non-oriented geodesic on $\mathbf{H}^2$ is specified by a set of distinct points on the boundary at infinity $\partial^\infty \mathbf{H}^2 = S^1$. It follows that measured geodesic laminations on  diffeomorphic hyperbolic surfaces may be compared by passing to the boundary at infinity of their universal covers. Thus, the space of measured geodesic laminations on $X$ depends only on the underlying topological surface $S_{g,n}$. \\

We denote the space of measured geodesic laminations on $S_{g,n}$ by $\mathcal{ML}_{g,n}$. It can be topologized by embedding it into the space of geodesic currents on $S_{g,n}$. By taking  geodesic representatives, integral multi-curves on $S_{g,n}$ can be interpreted as elements of $\mathcal{ML}_{g,n}$; we denote them by $\mathcal{ML}_{g,n}(\mathbf{Z})$. Given any marked hyperbolic structure $(X,\phi) \in \mathcal{T}_{g,n}$, there is a unique continuous affine extension of the length function $\ell_\cdot(X) \colon \mathcal{ML}_{g,n}(\mathbf{Z}) \to \mathbf{R}_{>0}$ to the set of all measured geodesic laminations on $S_{g,n}$; we also denote such extension by $\ell_\cdot(X) \colon \mathcal{ML}_{g,n} \to \mathbf{R}_{>0}$. For more details on the theory of measured geodesic laminations see \cite{Bon88}, \cite{Bon97}, and \cite[\S 8.3]{Mar16}.\\

\textit{The Thurston measure.}  The space of measured geodesic laminations $\mathcal{ML}_{g,n}$ admits a $6g-6+2n$ dimensional piecewise integral linear structure induced by train track charts. The integer points of this structure are precisely the integral multi-curves $\mathcal{ML}_{g,n}(\mathbf{Z}) \subseteq \mathcal{ML}_{g,n}$. For each $L > 0$, consider the counting measure $\mu^L$ on $\mathcal{ML}_{g,n}$ given by
\begin{equation}
\label{ML_counting_measure}
\mu^L := \frac{1}{L^{6g-6+2n}} \sum_{\gamma \in \mathcal{ML}_{g,n}(\mathbf{Z})} \delta_{ \frac{1}{L} \cdot \gamma}.
\end{equation}
As $L \to \infty$, this sequence of counting measures converges to a non-zero, locally finite measure $\mu_{\text{Thu}}$ on $\mathcal{ML}_{g,n}$ called the Thurston measure. This measure is $\text{Mod}_{g,n}$-invariant and belongs to the Lebesgue measure class. It also satisfies the following scaling property: $\mu_{\text{Thu}}(t \cdot A) = t^{6g-6+2n} \cdot \mu_{\text{Thu}}(A)$ for every measurable set $A \subseteq \mathcal{ML}_{g,n}$ and every $t > 0$.\\

Train track charts also induce a $\text{Mod}_{g,n}$-invariant symplectic form $\omega_{\text{Thu}}$ on $\mathcal{ML}_{g,n}$ called the Thurston symplectic form. For more details on the definition of $\omega_{\text{Thu}}$ see\cite[\S 3.2]{PH92}. The top exterior power $v_{\text{Thu}} := \frac{1}{(3g-3+n)!} \bigwedge^{3g-3+n}\omega_{\text{Thu}}$ is called the Thurston volume form. In \cite{Mas85}, Masur showed that the action of $\text{Mod}_{g,n}$ on $\mathcal{ML}_{g,n}$ is ergodic with respect to $\mu_{\text{Thu}}$. As a consequence, $\mu_{\text{Thu}}$ is the unique (up to scaling) $\text{Mod}_{g,n}$-invariant measure on $\mathcal{ML}_{g,n}$ in the Lebesgue measure class. It follows that the measure induced by the Thurston volume form on $\mathcal{ML}_{g,n}$ is a multiple of $\mu_{\text{Thu}}$. Moreover, see \cite{MT19} for a detailed proof, the scaling factor relating these measures can be computed explicitely:\\

\begin{proposition}\cite{MT19}
	\label{Thurston_measure_scaling_factor}
	If $\nu_{\text{Thu}}$ denotes the measure induced by the Thurston volume form on $\mathcal{ML}_{g,n}$, then
	\[
	\nu_{\text{Thu}} = 2^{2g-3+n} \cdot \mu_{Thu}.
	\]
\end{proposition}
$ $

\textit{Dehn-Thurston coordinates.}  Let $N := 3g-3+n$ and $\mathcal{P} := \{\gamma_1,\dots,\gamma_N\}$ be a pants decomposition of $S_{g,n}$. The following theorem, originally due to Dehn, gives an explicit parametrization of the set of integral multi-curves on $S_{g,n}$ in terms of their intersection numbers $m_i$ and their twisting numbers $t_i$ with respect to the curves $\gamma_i$ in $\mathcal{P}$; see \S 1.2 in \cite{PH92} for details.\\

\begin{theorem}
	\label{Dehn_Thurston_coordinates_int}
	There is a parametrization of $\mathcal{ML}_{g,n}(\mathbf{Z})$ by an additive semigroup $\Lambda \subseteq (\mathbf{Z}_{\geq 0} \times \mathbf{Z})^N$. The parameters $(m_i,t_i)_{i=1}^N \in (\mathbf{Z}_{\geq 0} \times \mathbf{Z})^N$ belong to $\Lambda$ if and only if the following conditions are satisfied:
	\begin{enumerate}
		\item For each $i =1,\dots, N$, if $m_i = 0$ then $t_i \geq 0$.
		\item For each complementary region $R$ of $S_{g,n} \backslash \mathcal{P}$, the parameters $m_i$ whose indices correspond to curves $\gamma_i$ of $\mathcal{P}$ bounding $R$ add up to an even number.\\
	\end{enumerate}
\end{theorem}

We refer to any parametrization as in Theorem \ref{Dehn_Thurston_coordinates_int} as a set of \textit{Dehn-Thurston coordinates} of $\mathcal{ML}_{g,n}(\mathbf{Z})$ adapted to $\mathcal{P}$ and to the additive semigroup $\Lambda\subseteq  (\mathbf{Z}_{\geq 0} \times \mathbf{Z})^N$ as the parameter space of such parametrization. By work of Thurston, see 8.3.9 in \cite{Mar16} for details, any set of Dehn-Thurston coordinates of $\mathcal{ML}_{g,n}(\mathbf{Z})$ extends to a parametrization of the whole space $\mathcal{ML}_{g,n}$ of measured geodesic laminations on $S_{g,n}$ in the following sense:\\

\begin{theorem}
	\label{Dehn_Thurston_coordinates_ML}
	Any set of Dehn-Thurston coordinates $(m_i,t_i)_{i=1}^N$ of $\mathcal{ML}_{g,n}(\mathbf{Z})$ with parameter space $\Lambda \subseteq (\mathbf{Z}_{\geq 0} \times \mathbf{Z})^N$ can be extended to a parametrization of $\mathcal{ML}_{g,n}$ by the set
	\[
	\Theta := \left\lbrace(m_i,t_i) \in (\mathbf{R}_{\geq 0} \times \mathbf{R})^N \ | \ m_i = 0 \Rightarrow t_i \geq 0, \ \forall i=1,\dots, N \right\rbrace.
	\]
\end{theorem}
$ $

We refer to any parametrization as in Theorem \ref{Dehn_Thurston_coordinates_ML} as a set of \textit{Dehn-Thurston coordinates} of $\mathcal{ML}_{g,n}$ adapted to $\mathcal{P}$ and to the set $\Theta \subseteq  (\mathbf{R}_{\geq 0} \times \mathbf{R})^N$ as the parameter space of such parametrization. For any such parametrization, the action of the full right Dehn twist along the cuff $\gamma_i$ of $\mathcal{P}$ on $\mathcal{ML}_{g,n}$ can be described in coordinates as $t_i \mapsto t_i + m_i$, leaving the other parameters constant. \\

Notice that the additive semigroup $\Lambda \subseteq (\mathbf{Z}_{\geq 0} \times \mathbf{Z})^N$ has index $2^{2g-3+n}$. Indeed, there is one even condition imposed on $\Lambda$ for every complementary region of $S_{g,n} \backslash \mathcal{P}$, of which there are $2g-2+n$ in total, and one of these conditions is redundant. It follows that the Thurston measure $\mu_{\text{Thu}}$ on $\mathcal{ML}_{g,n}$ corresponds to $2^{-(2g-3+n)}$ times the standard Lebesgue measure on $\Theta$.\\

\textit{Mirzakhani's integration formulas.} We briefly review Mirzakhani's integration formulas; see \cite{Mir07a}, \cite{Mir07c}, and \cite{Mir08b} for details, or \cite{Wol10} for a unified discussion. We will need to consider moduli spaces of hyperbolic surfaces with geodesic boundary. Let $g,n \geq 0$ be integers such that $2 - 2g - n < 0$ and $\mathbf{b} := (b_1,\dots,b_n) \in \mathbf{R}^n$ be a vector with non-negative entries. We denote by $\mathcal{M}_{g,n}(\mathbf{b})$ the moduli space of connected, oriented, complete, finite area hyperbolic surfaces of genus $g$ with $n$ labeled geodesic boundary components of lengths $b_1,\dots,b_n$; if $b_i = 0$ for some $i \in \{1, \dots, n\}$, we interpret the corresponding boundary component as a cusp. Just as in the case of surfaces without boundary, these moduli spaces carry natural Weil-Petersson volume forms. The following result, corresponding to Theorem 4.2 in \cite{Mir08b}, describes the behavior of the total Weil-Petersson volume of these moduli spaces as a function of the lengths $b_1,\dots,b_n$ of the boundary components.\\

\begin{theorem}\cite[Theorem 4.2]{Mir08b}
	\label{theo:vol pol}
	Let $g,n \geq 0$ be integers such that $2 - 2g - n < 0$. For vectors $\mathbf{b} := (b_1,\dots,b_n) \in \mathbf{R}^n$ with non-negative entries, the total Weil-Petersson volume 
	\[
	V_{g,n}(b_1,\dots,b_n) := \text{Vol}_{\text{wp}}(\mathcal{M}_{g,n}(\mathbf{b}))
	\]
	of the moduli space $\mathcal{M}_{g,n}(\mathbf{b})$ is a polynomial in $b_1^2,\dots,b_n^2$ of degree $3g-3+n$ all of whose coefficients are positive.\\
\end{theorem}

Moduli spaces as the ones described above can also be defined for topological surfaces with several connected components; they correspond (up to taking finite covers) to the product of the moduli spaces of the components of the surface. The volume polynomial of the corresponding moduli space is (up to a rational multiplicative factor) the product of the volume polynomials of the moduli spaces of the components of the surface.\\

Consider an ordered topological multi-curve $\gamma := (\gamma_1,\dots,\gamma_k)$ on $S_{g,n}$. We denote by $S_{g,n}(\gamma)$ the topological surface obtained by cutting $S_{g,n}$ along $\gamma$; it can have several connected components. For any vector $(x_1,\dots,x_k) \in \mathbf{R}^n$ with non-negative entries, we denote by $\mathcal{M}_{g,n}(\gamma, \mathbf{x})$ the moduli space of all oriented, complete, finite volume hyperbolic structures on $S_{g,n}(\gamma)$ with geodesic boundary components whose lengths are given by $x_1,\dots,x_k$ according to which curve $\gamma_i$ the boundary component comes from. We also denote by $V_{g,n}(\gamma,\mathbf{x})$ the total Weil-Petersson volume of the moduli space $\mathcal{M}_{g,n}(\gamma, \mathbf{x})$.\\

Given an ordered topological multi-curve $\gamma = (\gamma_1,\dots,\gamma_k)$ on $S_{g,n}$ and positive weights $\mathbf{a} := (a_1,\dots,a_k) \in \mathbf{R}^n$, we denote by $\mathbf{a} \cdot \gamma$ the unordered weighted multi-curve on $S_{g,n}$ given by
\[
\mathbf{a} \cdot \gamma = a_1 \gamma_1 + \cdots + a_k \gamma_k.
\]
The following theorem, corresponding to Theorem 4.1 in \cite{Mir08b}, gives a formula for the integral 
\[
P(L,\mathbf{a} \cdot \gamma) := \int_{\mathcal{M}_{g,n}}  \ s(X,\mathbf{a} \cdot \gamma,L) \ d\widehat{\mu}_{\text{wp}}(X)
\]
in terms of the volume polynomial $V_{g,n}(\gamma,\mathbf{x})$ of the moduli spaces $\mathcal{M}_{g,n}(\gamma,\mathbf{x})$ associated to the surface obtained by cutting $S_{g,n}$ along $\gamma$.\\

\begin{theorem}\cite[Theorem 4.1]{Mir08b}
	\label{theo:count_integral}
	For any ordered topological multi-curve $\gamma := (\gamma_1,\dots,\gamma_k)$ on $S_{g,n}$ and any positive weights $\mathbf{a} := (a_1,\dots,a_k) \in \mathbf{R}^k$, the integral over $\mathcal{M}_{g,n}$ of $s(X,\mathbf{a} \cdot \gamma,L)$ is given by
	\[
	P(L, \mathbf{a} \cdot \gamma) = \kappa(\gamma, \mathbf{a}) \cdot \int_{\mathbf{a} \cdot \mathbf{x} \leq T} V_{g,n}(\gamma,\mathbf{x}) \ \mathbf{x} \cdot d \mathbf{x},
	\]
	where $\mathbf{x} = x_1 \cdots x_k$, $d\mathbf{x} = dx_1 \cdots dx_k$, and $\kappa(\gamma,\mathbf{a}) \in \mathbf{Q}_{>0}$ is a constant depending only on $\gamma$ and $\mathbf{a}$ and taking only finitely many values as $\mathbf{a}$ varies.\\
\end{theorem}

Proposition \ref{prop:mir_freq} follows directly from Theorems \ref{theo:vol pol} and  \ref{theo:count_integral}.\\

\section{Square-integrability of the Mirzakhani function}

$ $

\textit{Notation}. For the rest of this paper, $N := 3g-3+n$ will denote the number of connected components of a pair of pants decomposition of $S_{g,n}$.\\ 

\textit{The Mirzakhani function near the cusp.} Let us first review Mirzakhani's original description of the behavior of the function $B \colon \mathcal{M}_{g,n} \to \mathbf{R}$ near the cusp. Recall the definition of the function $R \colon \mathbf{R}_+ \to \mathbf{R}_+$ in (\ref{eq:fnR}):
\[
R(x) = \frac{1}{x \cdot | \text{log}(x)|}.
\]
The following bounds, which describe the values of $B(X)$ for points $X \in \mathcal{M}_{g,n}$ near the cusp in terms of the lengths of short simple closed geodesics, are a direct consequence of Proposition 3.6 in \cite{Mir08b}:\\

\begin{proposition}
	\label{mir_original_bound}
	For all sufficiently small $\epsilon >0$, there are constants $C_1, C_2 > 0$ such that for all $X \in \mathcal{M}_{g,n}$,
	\[
	C_1 \cdot \prod_{\gamma \colon \ell_{\gamma}(X) \leq \epsilon} R(\ell_\gamma(X)) \leq	B(X) \leq C_2 \cdot \prod_{\gamma \colon \ell_{\gamma}(X) \leq \epsilon} \frac{1}{\ell_\gamma(X)}.
	\]
	where the products range over all simple closed geodesics $\gamma$ in $X$ of length $\leq \epsilon$.
\end{proposition}
$ $

Theorem \ref{theo:main_res_1} corresponds to the following improvement of the upper bound in Proposition \ref{mir_original_bound}:\\

\begin{theorem}
	\label{theo:mir_new_bound}
	For all sufficiently small $\epsilon >0$, there is a constant $C > 0$ such that for all $X \in \mathcal{M}_{g,n}$,
	\[
	B(X) \leq C \cdot \prod_{\gamma \colon \ell_{\gamma}(X) \leq \epsilon} R(\ell_\gamma(X)).
	\]
	where the product ranges over all simple closed geodesics $\gamma$ in $X$ of length $\leq \epsilon$.
\end{theorem}
$ $

The proof of Theorem \ref{theo:mir_new_bound} follows similar arguments as the ones given by Mirzakhani in the proof of Proposition \ref{mir_original_bound}; more precise estimates are considered when working with the Thurston measure.\\ 

Let us introduce some of the relevant terminology and tools used by Mirzakhani in the proof of Proposition \ref{mir_original_bound}. Fix a pair of pants decomposition $\mathcal{P} := \{\gamma_1,\dots,\gamma_N\}$ of $S_{g,n}$ and let $(m_i,t_i)_{i=1}^N$ be a set of Dehn-Thurston coordinates of $\mathcal{ML}_{g,n}(\mathbf{Z})$ adapted to $\mathcal{P}$; we denote by $\Lambda \subseteq (\mathbf{Z}_{\geq 0} \times \mathbf{Z})^N$ its parameter space and by $(m_i(\gamma),t_i(\gamma))_{i=1}^N$ the coordinates of any integral multi-curve $\gamma \in \mathcal{ML}_{g,n}(\mathbf{Z})$. Given an integral multi-curve $\gamma \in \mathcal{ML}_{g,n}(\mathbf{Z})$ and a marked hyperbolic structure $X \in \mathcal{T}_{g,n}$, we define the combinatorial length of $\gamma$ on $X$ with respect to the pants decomposition $\mathcal{P}$ to be
\begin{equation}
\label{eq:comb_length}
L_{\mathcal{P}}(X,\gamma) := \sum_{i=1}^N (m_i(\gamma) \cdot w(\ell_{\gamma_i}(X)) + |t_i(\gamma)| \cdot \ell_{\gamma_i}(X)),
\end{equation}
where $w \colon \mathbf{R}_{>0} \to \mathbf{R}_{>0}$ is the function
\[
w(x) := \arcsinh \left( \frac{1}{\sinh\left(\frac x 2 \right)}\right)
\]
describing the width of the hyperbolic collar neighborhoods introduced in Theorem \ref{theo:collar}. This definition depends on the choice of Dehn-Thurston coordinates considered.\\

Given $L > 0$, a pair of pants decomposition $\mathcal{P} := \{\gamma_1,\dots,\gamma_N\}$ of $S_{g,n}$, and a marked hyperbolic structure $X \in \mathcal{T}_{g,n}$, we say that $\mathcal{P}$ is $L$-bounded on $X$ if
\[
\ell_{\gamma_i}(X) \leq L, \ \forall i = 1,\dots,N.
\]
$ $

The main tool used by Mirzakhani in the proof of Proposition $\ref{mir_original_bound}$ is the following length comparison lemma, which corresponds to Proposition 3.5 in \cite{Mir08b}:\\

\begin{lemma}\cite[Proposition 3.5]{Mir08b}
	\label{length_comparison_lemma}
	Fix $L > 0$. There is a constant $C > 0$ (depending on $L$) such that for every $X \in \mathcal{T}_{g,n}$ and every pair of pants decomposition $\mathcal{P}$ of $S_{g,n}$ which is $L$-bounded on $X$, there is a set of Dehn-Thurston coordinates $(m_i,t_i)_{i=1}^N$ of $\mathcal{ML}_{g,n}(\mathbf{Z})$ adapted to $\mathcal{P}$ such that for every integral multi-curve $\gamma \in \mathcal{ML}_{g,n}(\mathbf{Z})$ the following bounds hold:
	\[
	\frac{1}{C} \cdot L_{\mathcal{P}}(X,\gamma) \leq  \ell_\gamma(X) \leq C \cdot L_{\mathcal{P}}(X,\gamma).
	\]
\end{lemma}
$ $

Recall that any set of Dehn-Thurston coordinates $(m_i,t_i)_{i=1}^N$ of $\mathcal{ML}_{g,n}(\mathbf{Z})$ with parameter space $\Lambda \subseteq (\mathbf{Z}_{\geq 0} \times \mathbf{Z})^N$ can be extended to give a parametrization of the space $\mathcal{ML}_{g,n}$ of measured geodesic laminations on $S_{g,n}$ by the set
\begin{equation}
\label{eq:parameters}
\Theta := \left\lbrace(m_i,t_i) \in (\mathbf{R}_{\geq 0} \times \mathbf{R})^N \ | \ m_i = 0 \Rightarrow t_i \geq 0, \ \forall i=1,\dots, N \right\rbrace.
\end{equation}
In particular, it is possible to define the combinatorial length of any measured geodesic lamination $\lambda \in \mathcal{ML}_{g,n}$ using (\ref{eq:comb_length}); we will also denote such combinatorial length by $L_{\mathcal{P}}(X,\lambda)$. \\

Recall that for every $X \in \mathcal{T}_{g,n}$ the hyperbolic length function $\ell_\cdot(X) \colon \mathcal{ML}_g \to \mathbf{R}_{>0}$ is homogeneous with respect to positive scalings. As any parametrization of $\mathcal{ML}_{g,n}$ by Dehn-Thurston coordinates is homogeneous with respect to positive scalings, it follows directly from the definition (\ref{eq:comb_length}) that any combinatorial length function $L_{\mathcal{P}}(X,\cdot) \colon \mathcal{ML}_{g,n} \to \mathbf{R}_{>0}$ is also homogeneous with respect to positive scalings. In particular, Lemma \ref{length_comparison_lemma} also holds for weighted multi-curves. As weighted multi-curves are dense in $\mathcal{ML}_{g,n}$ and as both the hyperbolic and combinatorial length functions are continuous, we deduce the following corollary:\\

\begin{corollary}
	\label{length_comparison_lemma_ML}
	Fix $L > 0$. There is a constant $C > 0$ (depending on $L$) such that for every $X \in \mathcal{T}_{g,n}$ and every pair of pants decomposition $\mathcal{P}$ of $S_{g,n}$ which is $L$-bounded on $X$, there is a set of Dehn-Thurston coordinates $(m_i,t_i)_{i=1}^N$ of $\mathcal{ML}_{g,n}$ adapted to $\mathcal{P}$ such that for every measured geodesic lamination $\lambda \in \mathcal{ML}_{g,n}$ the following bounds hold:
	\[
	\frac{1}{C} \cdot L_{\mathcal{P}}(X,\lambda) \leq  \ell_\lambda(X) \leq C \cdot L_{\mathcal{P}}(X,\lambda).
	\]
\end{corollary}
$ $

We are now ready to prove Theorem \ref{theo:mir_new_bound}:\\

\begin{proof}[Proof of Theorem \ref{theo:mir_new_bound}]
	Let $0 < \epsilon < 1$ be small enough so that on any hyperbolic surface no two simple closed geodesics of length $\leq \epsilon$ intersect. Consider an arbitrary hyperbolic surface $X \in \mathcal{M}_{g,n}$. After choosing an arbitrary marking, we can consider $X$ as a point in $\mathcal{T}_{g,n}$. Let $\{\gamma_1,\dots,\gamma_k\}$ be the set of all simple closed curves on $S_{g,n}$ having length $\leq \epsilon$ on $X$. Notice that the choice of $\epsilon > 0$ forces these simple closed curves to be pairwise disjoint; in particular $0 \leq k \leq N$. By Theorem \ref{theo:bers}, we can complete the collection $\{\gamma_1,\dots,\gamma_k\}$ to a pair of pants decomposition 
	\[
	\mathcal{P}_X := \{\gamma_1,\dots,\gamma_k,\gamma_{k+1},\dots,\gamma_N\}
	\]
	of $S_{g,n}$ satisfying 
	\[
	\ell_{\gamma_i}(X) \leq L_{g,n}, \ \forall i=1,\dots,N,
	\] 
	where $L_{g,n} > 1$ is a constant depending only on $g$ and $n$. In other words, $\mathcal{P}_X$ is $L_{g,n}$-bounded on $X$.\\
	
	Consider the subsets $B_X,B_{X,\mathcal{P}_X} \subseteq \mathcal{ML}_{g,n}$ given by
	\begin{align*}
	B_X &:= \{\lambda \in \mathcal{ML}_{g,n} \ | \ \ell_X(\lambda) \leq 1 \},\\
	B_{X,\mathcal{P}_X} &:= \{\lambda \in \mathcal{ML}_{g,n} \ | \ L_{\mathcal{P}_X}(X,\lambda) \leq 1 \},
	\end{align*}
	where the set of Dehn-Thurston coordinates used to define $L_{\mathcal{P}_X}(X,\cdot)$ is the one given by Corollary \ref{length_comparison_lemma_ML}. It follows from Corollary \ref{length_comparison_lemma_ML} that
	\[
	B_X \subseteq C \cdot B_{X,\mathcal{P}_X}
	\]
	for some constant $C > 0$ depending only on $g$, $n$, and $L_{g,n}$. Using the scaling properties of the Thurston measure we deduce
	\[
	B(X) = \mu_{\text{Thu}}(B_X) \leq \mu_{\text{Thu}}(C \cdot B_{X,\mathcal{P}_X}) = C^{2N} \cdot \mu_{\text{Thu}}( B_{X,\mathcal{P}_X}).
	\]
	This reduces our problem to computing $\mu_{\text{Thu}}( B_{X,\mathcal{P}_X})$.\\
	
	We compute $\mu_{\text{Thu}}( B_{X,\mathcal{P}_X})$ explicitely using Dehn-Thurston coordinates. Recall that for any set of Dehn-Thurston coordinates $(m_i,t_i)_{i=1}^N$ of  $\mathcal{ML}_{g,n}$ with parameter space $\Theta$ as in (\ref{eq:parameters}), the Thuston measure $\mu_{\text{Thu}}$ on $\mathcal{ML}_{g,n}$ corresponds to $2^{g-N}$ times the standard Lebesgue measure on $\Theta$, which we denote by $\text{Leb}$. It follows that
	\[
	\mu_{\text{Thu}}( B_{X,\mathcal{P}_X}) = 2^{g-N} \cdot \text{Leb}(A_{X,\mathcal{P}_X}),
	\]
	where
	\[
	A_{X,\mathcal{P}_X} := \left\lbrace (m_i,t_i) \in \Theta \ \bigg| \ \sum_{i=1}^N (m_i \cdot w(\ell_{\gamma_i}(X)) + |t_i| \cdot \ell_{\gamma_i}(X)) \leq 1 \right\rbrace.
	\]
	A direct calculation shows that
	\[
	\text{Leb}(A_{X,\mathcal{P}_X}) =  \frac{2^N}{N!} \cdot \prod_{i=1}^N \frac{1}{\ell_{\gamma_i}(X) \cdot w(\ell_{\gamma_i}(X))}.
	\]
	Recall
	\[
	\lim_{x \to 0^+} \frac{w(x)}{|\log(x)|} = 1.
	\]
	As a consequence, for every sufficiently small $0 < \epsilon < 1$ and every $0 < x < \epsilon$,
	\[
	w(x) \geq \frac{1}{2} \cdot  |\log(x)|.
	\]
	In particular, for every $i \in \{1,\dots,k\}$ we can bound
	\[
	\frac{1}{\ell_{\gamma_i}(X) \cdot w(\ell_{\gamma_i}(X))} \leq 
	2 \cdot R(\ell_{\gamma_i}(X)).
	\]
	Consider the function $H \colon \mathbf{R}_{>0} \to \mathbf{R}_{>0}$ defined as
	\[
	H(x) := \frac{1}{x \cdot w(x)}.
	\]
	Given $\epsilon > 0$ sufficiently small, let $M > 0$ be the maximum attained by $H$ on the compact interval $[\epsilon, L_{g,n}]$. For every $i \in \{k+1,\dots,N\}$ we can bound
	\[
	\frac{1}{\ell_{\gamma_i}(X) \cdot w(\ell_{\gamma_i}(X))} \leq M.
	\]
	Putting everything together we deduce
	\[
	B(X) \leq \frac{C^{2N} \cdot 2^{g+k} \cdot M^{N-k}}{N!} \cdot \prod_{i=1}^k R(\ell_X(\gamma_i)),
	\]
	finishing the proof.
\end{proof}
$ $

\begin{remark}
	The proof of Theorem \ref{theo:mir_new_bound} shows that \textit{how sufficiently small} the values of $\epsilon > 0$ considered need to be is independent of $g$ and $n$.\\
\end{remark}

\textit{Square-integrability of the Mirzakhani function.} Fix $0 < \epsilon < 1$ small enough according to Theorem \ref{theo:mir_new_bound}. It follows from Theorem \ref{theo:mir_new_bound} that the integrability properties of the function $F \colon \mathcal{M}_{g,n} \to \mathbf{R}_{>0}$ given by
\begin{equation}
\label{eq:F}
F(X) := \prod_{\gamma \colon \ell_{\gamma}(X) \leq \epsilon} R(\ell_X(\gamma)),
\end{equation}
where the product ranges over all simple closed geodesics $\gamma$ in $X$ of length $\leq \epsilon$, are inherited by the Mirzakhani function. Motivated by this idea we prove the following result.\\ 

\begin{proposition}
	\label{prop:F_L2}
	The function $F \colon \mathcal{M}_{g,n} \to \mathbf{R}_{>0}$ defined in (\ref{eq:F}) is square integrable with respect to the Weil-Petersson volume form on $\mathcal{M}_{g,n}$, i.e.,
	\[
	\int_{\mathcal{M}_{g,n}} F(X)^2 \ d\widehat{\mu}_{\text{wp}}(X) < +\infty
	\]
\end{proposition}
$ $

\begin{proof}
	For every $k \in \{0,\dots, N\}$, let $\mathcal{M}_{g,n}^{k,\epsilon} \subseteq \mathcal{M}_{g,n}$ be the subset of all the hyperbolic surfaces in $\mathcal{M}_{g,n}$ with exactly $k$ simple closed geodesics of length $\leq \epsilon$ (by Mumford's compactness criterion, see  12.4 in \cite{FM11} for instance, $\mathcal{M}_{g,n}^{0,\epsilon}$ is compact). It is enough for our purposes to show that for every $k \in \{0,\dots, N\}$ the following integral is finite:
	\[
	\int_{\mathcal{M}_{g,n}^{k,\epsilon}} F(X)^2 \ d\widehat{\mu}_{\text{wp}}(X).
	\]
	$ $
	
	Fix $k \in \{0,\dots,N\}$. Let $L_{g,n} > 1$ be as in Theorem \ref{theo:bers}. As consequence of Theorem \ref{theo:bers} and of the fact that there are only finitely many pair of pants decompositions of $S_{g,n}$ up to the action of the mapping class group, we see that $\mathcal{M}_{g,n}^{k,\epsilon}$ can be covered by finitely many subsets of $\mathcal{T}_{g,n}$ which in appropriate Fenchel-Nielsen coordinates $(\ell_i,\tau_i)_{i=1}^N \in (\mathbf{R}_{>0} \times \mathbf{R})^N$ are given by
	\[
	A_{g,n}^{k,\epsilon}:= \left\lbrace
	\begin{array}{c | l}
	(\ell_i,\tau_i)_{i=1}^N \in (\mathbf{R}_{> 0} \times \mathbf{R})^N
	& \ 0 \leq \tau_i < \ell_i, \ \forall i=1,\dots,N,\\
	& \ 0 <\ell_i \leq \epsilon,  \ \forall i=1,\dots,k,\\
	& \ \epsilon < \ell_i \leq L_{g,n},  \ \forall i=k+1,\dots,N.\\
	\end{array} \right\rbrace,
	\]
	Let $\mathcal{A}_{g,n}^{k,\epsilon} \subseteq \mathcal{T}_{g,n}$ be one of these subsets. It is enough for our purposes to show that 
	\[
	\int_{\mathcal{A}_{g,n}^{k,\epsilon}} \widetilde{F}(X)^2 \ d\mu_{\text{wp}}(X) < +\infty,
	\]
	where $\widetilde{F} \colon \mathcal{T}_{g,n} \to \mathbf{R}_{>0}$ denotes the lift of $F$ to $\mathcal{T}_{g,n}$.\\
	
	Using Wolpert's magic formula we compute
	\begin{align*}
	\int_{\mathcal{A}_{g,n}^{k,\epsilon}} \widetilde{F}(X)^2 \ d\mu_{\text{wp}}(X) &= \int_{A_{g,n}^{k,\epsilon}} \prod_{i=1}^k \frac{1}{\ell_i^2 \cdot\log(\ell_i)^2} \ d\tau_1 \cdots d\tau_N \ d\ell_1 \cdots d\ell_N\\
	&= \left(\prod_{i=1}^k \int_0^\epsilon \int_0^{\ell_i} \frac{1}{\ell_i^2 \cdot \log(\ell_i)^2} \ d\tau_i \ d\ell_i\right) \\
	&\quad \cdot \left(\prod_{i=k+1}^{N} \int_\epsilon^{L_{g,n}} \int_\epsilon^{\ell_i}  \ d\tau_i \ d\ell_i \right)
	\end{align*}
	Direct computations show
	\begin{align*}
	\int_0^\epsilon \int_0^{\ell_i} \frac{1}{\ell_i^2 \cdot \log(\ell_i)^2} \ d\tau_i \ d\ell_i &=  \frac{-1}{\log(\epsilon)} < +\infty, \\
	\int_\epsilon^{L_{g,n}} \int_0^{\ell_i}  \ d\tau_i \ d\ell_i &= \frac{L_{g,n}^2 - \epsilon^2}{2} < +\infty.
	\end{align*}
	It follows that
	\[
	\int_{\mathcal{A}_{g,n}^{k,\epsilon}} \widetilde{F}(X)^2 \ d\mu_{\text{wp}}(X) < +\infty,
	\]
	completing the proof.
\end{proof}
$ $

As a direct consequence of Theorem \ref{theo:mir_new_bound} and Proposition \ref{prop:F_L2} we deduce:\\

\begin{theorem}
	\label{theo:mir_L2}
	The Mirzakhani function $B \colon \mathcal{M}_{g,n} \to \mathbf{R}_{>0}$ is square integrable with respect to the Weil-Petersson volume form on $\mathcal{M}_{g,n}$, i.e.,
	\[
	a_{g,n} := \int_{\mathcal{M}_{g,n}} B(X)^2 \ d\widehat{\mu}_{\text{wp}}(X) < +\infty
	\]
\end{theorem}
$ $

\begin{remark}
	\label{rem:B_notin_L3}
	Using the lower bound in Proposition \ref{mir_original_bound} and computations similar to the ones in the proof of Proposition \ref{prop:F_L2}, one can show $B \notin L^{2+\epsilon}(\mathcal{M}_{g,n}, \widehat{\mu}_{\text{wp}})$ for every $\epsilon > 0$. \\
\end{remark}

\section{Statistics of counting problems for simple closed geodesics}

$ $

\textit{Joint frequencies.} Let $\gamma_1,\gamma_2 \in \mathcal{ML}_{g,n}(\mathbf{Z})$ be a pair of integral multi-curves on $S_{g,n}$. Recall the definition of their \textit{joint frequency} $c(\gamma_1,\gamma_2)$ given in (\ref{eq:joint_freq}):
\begin{equation}
\label{eq:joint_freq_2}
c(\gamma_1,\gamma_2) := \lim_{L \to \infty} \frac{1}{L^{12g-12+4n}} \int_{\mathcal{M}_{g,n}} s(X,\gamma_1,L) \cdot s(X,\gamma_2,L) \ d\widehat{\mu}_{\text{wp}}(X).
\end{equation}
$ $

We now prove Theorem \ref{theo:main_res_4}, which we restate here for convenience:\\

\begin{theorem}
	\label{theo:joint_freq_rel}
	For every pair of integral multi-curves $\gamma_1,\gamma_2 \in \mathcal{ML}_{g,n}(\mathbf{Z})$, the limit in the definition (\ref{eq:joint_freq_2}) of $c(\gamma_1,\gamma_2)$ exists and moreover,
	\[
	c(\gamma_1,\gamma_2) = \frac{a_{g,n}}{b_{g,n}^2} \cdot c(\gamma_1) \cdot c(\gamma_2).
	\]
\end{theorem}
$ $

To prove Theorem \ref{theo:joint_freq_rel} we make use of the following upper bound, similar in spirit to the one in Theorem \ref{theo:mir_new_bound}.\\

\begin{proposition}
	\label{prop:count_L2_bound}
	For all sufficiently small $\epsilon >0$, there exist constants $C > 0$ and $L_0 > 0$ such that for all $L \geq L_0$, all $X \in \mathcal{M}_{g,n}$, and all $\eta \in \mathcal{ML}_{g,n}(\mathbf{Z})$, 
	\[
	\frac{s(X,\eta,L)}{L^{6g-6+2n}} \leq C \cdot \prod_{\gamma \colon \ell_{\gamma}(X) \leq \epsilon} R(\ell_X(\gamma)),
	\]
	where the product ranges over all simple closed geodesics $\gamma$ in $X$ of length $\leq \epsilon$.
\end{proposition}
$ $

\begin{proof}
	We proceed as in the proof of Theorem \ref{theo:mir_new_bound}. 	Let $0 < \epsilon < 1$ be small enough so that on any hyperbolic surface no two simple closed geodesics of length $\leq \epsilon$ intersect. Consider an arbitrary hyperbolic surface $X \in \mathcal{M}_{g,n}$. After choosing an arbitrary marking, we can consider $X$ as a point in $\mathcal{T}_{g,n}$. Let $\{\gamma_1,\dots,\gamma_k\}$ be the set of all simple closed curves on $S_{g,n}$ having length $\leq \epsilon$ on $X$. Notice that the choice of $\epsilon > 0$ forces these simple closed curves to be pairwise disjoint; in particular $0 \leq k \leq N$. By Theorem \ref{theo:bers}, we can complete the collection $\{\gamma_1,\dots,\gamma_k\}$ to a pair of pants decomposition 
	\[
	\mathcal{P}_X := \{\gamma_1,\dots,\gamma_k,\gamma_{k+1},\dots,\gamma_N\}
	\]
	of $S_{g,n}$ satisfying 
	\[
	\ell_{\gamma_i}(X) \leq L_{g,n}, \ \forall i=1,\dots,N,
	\] 
	where $L_{g,n} > 1$ is a constant depending only on $g$ and $n$. In other words, $\mathcal{P}_X$ is $L_{g,n}$-bounded on $X$.\\

	Fix $\eta \in \mathcal{ML}_{g,n}(\mathbf{Z})$. For every $L > 0$ we consider the counting functions
	\begin{align*}
	s(X,\eta,L) &:= \# \{\alpha \in \text{Mod}_{g,n} \cdot \eta \ | \ \ell_X(\alpha) \leq L \},\\
	S(X,\eta,L) &:= \# \{\alpha \in \text{Mod}_{g,n} \cdot \eta \ | \ L_{\mathcal{P}_X}(X,\alpha) \leq L \},
	\end{align*}
	where the set of Dehn-Thurston coordinates used to define $L_{\mathcal{P}_X}(X,\cdot)$ is the one given by Lemma \ref{length_comparison_lemma}. It follows from Lemma \ref{length_comparison_lemma} that for every $L > 0$,
	\[
	s(X,\eta,L) \leq S(X,\eta,CL),
	\]
	where $C > 0$ is a constant depending only on $g$, $n$, and $L_{g,n}$. This reduces our problem to giving appropriate upper bounds for the values of $S(X,\eta,CL)$ across all $L \geq L_0$, with $L_0 > 0$ depending only on $g$, $n$, and $L_{g,n}$.\\
	
	We bound the values of $S(X,\eta,CL)$ by using Dehn-Thurston coordinates. Let $(m_i,t_i)_{i=1}^N$ be the set of Dehn-Thurston coordiantes of $\mathcal{ML}_{g,n}(\mathbf{Z})$ used to define the combinatorial length $L_{\mathcal{P}_X}(X,\cdot)$ above and let $\Lambda \subseteq (\mathbf{Z}_{\geq 0} \times \mathbf{Z})^N$ be its parameter space. We denote by $\Lambda_{\eta} \subseteq \Lambda$ the set of all parameters in $\Lambda$ that represent integral multi-curves in $\text{Mod}_{g,n} \cdot \eta$. Notice that for every $L > 0$,
	\[
	S(X,\eta,CL) = \# \left\lbrace (m_i,t_i) \in \Lambda_{\eta} \ \bigg| \ \sum_{i=1}^N (m_i \cdot w(\ell_{\gamma_i}(X)) + |t_i| \cdot \ell_{\gamma_i}(X)) \leq CL \right\rbrace.
	\]
	One can bound $S(X,\eta,CL)$ by the standard Lebesgue measure of the box
	\[
	B_{CL}^N := \left\lbrace
	\begin{array}{c | l}
	(x_i,y_i)_{i=1}^N\in (\mathbf{R}_{\geq 0} \times \mathbf{R})^N
	& \ 0 \leq x_i < \frac{CL}{w(\ell_{\gamma_i}(X))} + 1, \ \forall i=1,\dots,N,\\[5pt]
	& \ 0 \leq |y_i| \leq \frac{CL}{\ell_{\gamma_i}(X)} + 1,  \ \forall i=1,\dots,N.\\
	\end{array} \right\rbrace,
	\]
	but this does not give an upper bound of the desired order when $\ell_{\gamma_i}(X) \ll CL \ll  w(\ell_{\gamma_i}(X))$ for some $i \in \{1,\dots,N\}$.  Roughly speaking, in the regime $0 < a \ll 1 \ll b$, the area of the thin rectangle $R \subseteq \mathbf{R}^2$ with vertices $(0,b)$, $(0,-b)$, $(a,b)$, and $(a,-b)$ is not a good approximation for the number of integer points in $R$. There is a simple way to get around this difficulty though.\\
	
	We make the following key observation: given $L > 0$, if $w(\ell_{\gamma_i}(X))  > CL$ for some $i \in \{1,\dots,N\}$, then none of the integral multi-curves counted by the function $S(X,\eta,CL)$ intersect $\gamma_i$; in terms of Dehn-Thurston coordinates, $m_i = 0$ for all such integral multi-curves. Points in $\Lambda$ with $i$-th coordinates of the form $(0,t_i)$ and $t_i > 0$ represent integral multi-curves on $S_{g,n}$ one of whose topological components is $\gamma_i$, with weight $t_i$. As $\gamma$ has at most $3g-3+n$ topological components, there are at most $3g-2+n$ distinct possible values such $t_i$ can take when describing curves in the mapping class group orbit of $\eta$ ($t_i = 0$ is allowed).\\
	
	Given $L > 0$, relabel the $\gamma_i$'s so that $w(\ell_{\gamma_i}(X))  \leq CL$ for all $i \in \{1,\dots,t\}$ and $w(\ell_{\gamma_i}(X)) >  CL$ for all $i \in \{t+1,\dots,N\}$; the index $t \in \{0,\dots,N\}$ depends on $L$. Clearly $t = N$ for all big enough $L$, but how big $L$ needs to be for such condition to hold depends on $X$. As we are looking for an upper bound uniform across all big enough values of $L$, it is important to keep track of the index $t$. In this context we consider the truncated counting function
	\[
	B_t(X,CL) := \# \left\lbrace (m_i,t_i) \in (\mathbf{Z}_{\geq 0} \times \mathbf{Z})^t \ \bigg| \ \sum_{i=1}^t (m_i \cdot w(\ell_{\gamma_i}(X)) + |t_i| \cdot \ell_{\gamma_i}(X)) \leq CL \right\rbrace.
	\]	
	It follows from the key observation above that
	\[
	S(X,\eta,CL) \leq (3g-2+n)^{N-t} \cdot B_t(X,CL).
	\]
	$ $
	
	Let $L_0 :=  L_{g,n}/C$ so that $\ell_{\gamma_i}(X) \leq CL$ for all $i \in \{1,\dots,N\}$ and all $L \geq L_0$. Fix $L \geq L_0$ and let $t \in \{0,\dots,N\}$ be as in the previous paragraph. The conditions 
	\begin{equation}
	\label{eq:cond}
	w(\ell_{\gamma_i}(X)) \leq CL, \quad \ell_{\gamma_i}(X) \leq CL, \quad \forall i = 1,\dots,t
	\end{equation}
	will allow us to get an upper bound of the desired order for $B_t(X,CL)$. Notice
	\[
	B_t(X,CL) \leq \text{Leb}(B_{CL}^t),
	\]
	where $\text{Leb}(B_{CL}^t)$ denotes the standard Lebesgue measure of the box
	\[
	B_{CL}^t := \left\lbrace
	\begin{array}{c | l}
	(x_i,y_i)_{i=1}^t\in (\mathbf{R}_{\geq 0} \times \mathbf{R})^t
	& \ 0 \leq x_i < \frac{CL}{w(\ell_{\gamma_i}(X))} + 1, \ \forall i=1,\dots,t,\\[5pt]
	& \ 0 \leq |y_i| \leq \frac{CL}{\ell_{\gamma_i}(X)} + 1,  \ \forall i=1,\dots,t.\\
	\end{array} \right\rbrace.
	\]
	A direct calculation together with the conditions in (\ref{eq:cond}) give
	\[
	\text{Leb}(B_{CL}^t) = \prod_{i=1}^t 2 \cdot \left( \frac{CL}{w(\ell_{\gamma_i}(X))} + 1\right) \cdot \left( \frac{CL}{\ell_{\gamma_i}(X)} + 1\right)\leq \prod_{i=1}^t \frac{8 \cdot C^2 \cdot L^2}{\ell_{\gamma_i}(X) \cdot w(\ell_{\gamma_i}(X))}.
	\]
	Putting things together we deduce
	\begin{equation}
	\label{eq:partial_s_bound}
	s(X,\eta,L)  \leq (3g-2+n)^{N-t} \cdot 8^t \cdot C^{2t} \cdot L^{2t}  \cdot \prod_{i=1}^t \frac{1}{\ell_{\gamma_i}(X) \cdot w(\ell_{\gamma_i}(X))}.
	\end{equation}
	$ $
	
	Notice that
	\[
	\lim_{x \to 0+} \frac{1}{x \cdot w(x)} = +\infty.
	\]
	As a consequence, we can find $\delta > 0$ such that for all $0 < x < \delta$,
	\[
	\frac{1}{x \cdot w(x)} \geq 1.
	\]
	Notice also that $w \colon \mathbf{R}_{>0} \to \mathbf{R}_{>0}$ is an orientation reversing homeomorphism. Therefore we can find $L_1 > 0$ such that for all $L \geq L_1$, if $w(\ell_{\gamma_i}(X)) > CL$ for some $i \in \{1,\dots,N\}$, then $\ell_{\gamma_i}(X) < \delta$. In particular, given $L \geq L_1$ and $t \in \{0,\dots,N\}$ as above, every $i \in \{t+1,\dots,N\}$ satisfies $\ell_{\gamma_i}(X) < \delta$, and so we can bound
	\[
	1 \leq \frac{1}{\ell_{\gamma_i}(X) \cdot w(\ell_{\gamma_i}(X))}.
	\]
	It follows from (\ref{eq:partial_s_bound}) that under the condition $L \geq \max\{L_0,L_1,1\}$ we have
	\[
	s(X,\eta,L)  \leq (3g-2+n)^{N} \cdot 8^N \cdot C^{2N} \cdot L^{2N}  \cdot \prod_{i=1}^N \frac{1}{\ell_{\gamma_i}(X) \cdot w(\ell_{\gamma_i}(X))},
	\]
	where we assume without loss of generality that $C > 1$.\\
	$ $
	
	Proceeding just as in the last part of the proof of Theorem \ref{theo:mir_new_bound}, one can get an upper bound for $s(X,\eta,L)$ depending only on the simple closed curves $\gamma_i$ with $i \in \{1,\dots,k\}$, finishing the proof.
\end{proof}
$ $

\begin{remark}
	\label{rem:b(X,L)_no_bound}
	For every $X \in \mathcal{M}_{g,n}$ and every $L > 0$ consider the counting function
	\[
	b(X,L):= \# \{\alpha \in \mathcal{ML}_{g,n}(\mathbf{Z}) \ | \ \ell_X(\alpha) \leq L \}.
	\]
	No upper bound as the one in Proposition \ref{prop:count_L2_bound} can be given for these counting functions. Indeed, the \textit{thin rectangle phenomenon} described in the proof of Proposition \ref{prop:count_L2_bound} can be used to show that $b(\cdot,L) \notin L^{2}(\mathcal{M}_{g,n},\widehat{\mu}_{\text{wp}})$ for every $L > 0$, in contrast with Proposition \ref{prop:F_L2}.\\
\end{remark}

\begin{remark}
	\label{rem:alt_proof}
	Fix $\gamma \in \mathcal{ML}_{g,n}(\mathbf{Z})$. Given $X \in \mathcal{M}_{g,n}$, Theorems \ref{theo:mir_count} and \ref{theo:lead} ensure
	\[
	\lim_{L \to \infty} \frac{b_{g,n}}{c(\gamma)} \cdot \frac{s(X,\gamma,L)}{L^{6g-6+2n}} = B(X).
	\]
	Fix $0 < \epsilon < 1$ small enough according to Proposition \ref{prop:count_L2_bound}. By Proposition \ref{prop:count_L2_bound} we can find a constant $C > 0$ depending only on $g$ and $n$ such that for all big enough $L >0$ and all $X \in \mathcal{M}_{g,n}$,
	\begin{equation}
	\label{eq:rmk1}
	\frac{b_g}{c(\gamma)} \cdot \frac{s(X,\gamma,L)}{L^{6g-6+2n}} \leq C \cdot F(X),
	\end{equation}
	where $F \colon \mathcal{M}_{g,n} \to \mathbf{R}_{>0}$ is the function defined in (\ref{eq:F}). Taking $L \to \infty$ in (\ref{eq:rmk1}) we deduce
	\[
	B(X) \leq C \cdot F(X)
	\]
	for all $X \in \mathcal{M}_{g,n}$. This argument yields an alternative proof of Theorem \ref{theo:mir_new_bound}.\\
\end{remark}

\begin{remark}
	Remark \ref{rem:B_notin_L3} and the arguments in Remark \ref{rem:alt_proof} show that the upper bound in Proposition \ref{prop:count_L2_bound} cannot be improved to attain more integrablity of the bounding function if we want the bound to hold uniformly for all big enough $L >0$.\\
\end{remark}

Theorem \ref{theo:joint_freq_rel} now easily follows from Theorems \ref{theo:mir_count} and \ref{theo:lead}, Proposition \ref{prop:count_L2_bound}, and the dominated convergence theorem.\\

\begin{proof}[Proof of Theorem \ref{theo:joint_freq_rel}]
	By Theorems \ref{theo:mir_count} and \ref{theo:lead} we have
	\[
	\lim_{L \to \infty} \frac{s(X,\gamma_i,L)}{L^{6g-6+2n}} = \frac{c(\gamma_i) \cdot B(X)}{b_{g,n}} \\
	\]
	for every $X \in \mathcal{M}_{g,n}$ and every $i \in \{1,2\}$. It follows that 
	\[
	\lim_{L \to \infty} \frac{s(X,\gamma_1,L) \cdot s(X,\gamma_2,L) }{L^{12g-12+4n}} = \frac{c(\gamma_1) \cdot c(\gamma_2)}{b_{g,n}^2} \cdot B(X)^2.
	\]
	for every $X \in \mathcal{M}_{g,n}$. Fix $0 < \epsilon < 1$ small enough according to Proposition \ref{prop:count_L2_bound}. By Proposition \ref{prop:count_L2_bound} we have
	\[
	\frac{s(X,\gamma_1,L) \cdot s(X,\gamma_2,L) }{L^{12g-12+4n}} \leq C \cdot F(X)^2
	\]
	for all big enough $L > 0$ and all $X \in \mathcal{M}_{g,n}$, where $C  > 0$ is a constant depending only on $g$ and $n$, and  $F \colon \mathcal{M}_{g,n} \to \mathbf{R}_{>0}$ is the function defined in (\ref{eq:F}). Proposition \ref{prop:F_L2} shows that $F^2 \in L^1(\mathcal{M}_{g,n},\widehat{\mu}_{\text{wp}})$. By the dominated convergence theorem, the limit in the definition (\ref{eq:joint_freq_2}) of $c(\gamma_1,\gamma_2)$ exists and moreover,
	\[
	c(\gamma_1,\gamma_2) = \frac{c(\gamma_1) \cdot c(\gamma_2)}{b_{g,n}^2} \cdot \int_{\mathcal{M}_{g,n}} B(X)^2 \ d\widehat{\mu}_{\text{wp}}(X) = \frac{a_{g,n}}{b_{g,n}^2} \cdot c(\gamma_1) \cdot c(\gamma_2).
	\]
	This finishes the proof.
\end{proof}
$ $

\textit{Recovering $b_{g,n}$ and $a_{g,n}$ from frequencies and joint frequencies.} Theorem 5.3 of \cite{Mir08b} establishes the following relation between the constant $b_{g,n}$ and the frequencies $c(\gamma)$:  \\

\begin{theorem}
	\label{theo:vol_freq}
	For any integers $g,n \geq 0$ such that $2 -2g -n < 0$,
	\[
	b_{g,n} = \sum_{\gamma \in \mathcal{ML}_{g,n}(\mathbf{Z})/\text{Mod}_{g,n}} c(\gamma).
	\]
\end{theorem}
$ $

The authors feel the need to include a proof of Theorem \ref{theo:vol_freq} as some details, most likely known to Mirzakhani herself, are omitted in \cite{Mir08b}. Our proof relies on the following rough estimate, which separates the dependence of the counting function $s(X,\mathbf{a} \cdot \gamma, L)$ on the hyperbolic structure $X \in \mathcal{M}_{g,n}$ and weight parameters $\mathbf{a} \in \mathbf{N}^k$.\\

\begin{lemma}
	\label{lem:count_bound}
	Let $X \in \mathcal{M}_{g,n}$ and $\gamma := (\gamma_1,\dots,\gamma_k)$ with $1 \leq k \leq N$ be an ordered unweighted multi-curve on $S_{g,n}$. There exist constants $C := C(X)>0$ and $L_0 > 0$ such that for all $\mathbf{a} := (a_1,\dots,a_k) \in \mathbf{N}^k$ and all $L \geq L_0$, the following bound holds:
	\[
	\frac{s(X, \mathbf{a} \cdot \gamma,L)}{L^{6g-6+2n}} \leq C \cdot \prod_{i=1}^k \frac{1}{a_i^2}.
	\]
\end{lemma}
$ $

\begin{proof}
	 The following proof makes strong use of Mirzakhani's integration formulas, in particular of Theorem \ref{theo:count_integral}; we refer the reader to the statement of such theorem for the notation used throughout the rest of this proof. Let $\mathbf{a} := (a_1,\dots,a_k) \in \mathbf{N}^k$ and $L > 0$ be arbitrary. According to Theorem \ref{theo:count_integral},
	\[
	\int_{\mathcal{M}_{g,n}} s(X, \mathbf{a} \cdot \gamma,L) \ d\widehat{\mu}_{\text{wp}}(X)  = \kappa(\gamma, \mathbf{a}) \cdot \int_{\mathbf{a} \cdot \mathbf{x} \leq L} V_{g,n}(\gamma,\mathbf{x}) \cdot \mathbf{x} \ d \mathbf{x}.
	\]
	For all $i \in \{1,\dots,k\}$ consider the change of variables $u_i := a_i x_i$, so that $du_i = a_i dx_i$. It follows from the change of variables formula that
	\[
	\int_{\mathbf{a} \cdot \mathbf{x} \leq L} V_{g,n}(\gamma,\mathbf{x}) \cdot \mathbf{x} \ d \mathbf{x} = \prod_{i=1}^k \frac{1}{a_i^2} \cdot\int_{\mathbf{1} \cdot \mathbf{u} \leq L} V_{g,n}(\gamma, \mathbf{u}/\mathbf{a}) \ \mathbf{u} \cdot d \mathbf{u},
	\]
	where $\mathbf{u} = u_1 \cdots u_k$, $d \mathbf{u} = du_1 \cdots du_k$, and $\mathbf{u}/\mathbf{a} = (u_1/a_1,\dots,u_k/a_k)$. As a consequence of Theorem \ref{theo:vol pol}, the function  $V_{g,n}(\gamma,\mathbf{x})$ is a polynomial in $\mathbf{x}$ with non-negative coefficients. It follows that (since each $a_i \geq 1$)
	\[
	V_{g,n}(\gamma, \mathbf{u}/\mathbf{a}) \leq V_{g,n}(\gamma, \mathbf{u})
	\]
	for all $\mathbf{u} \in \mathbf{R}^k$ with non-negative entries. In particular,
	\[
	\int_{\mathbf{1} \cdot \mathbf{u} \leq L} V_{g,n}(\gamma, \mathbf{u}/\mathbf{a}) \ \mathbf{u} \cdot d \mathbf{u}  \leq \int_{\mathbf{1} \cdot \mathbf{u} \leq L} V_{g,n}(\gamma, \mathbf{u}) \cdot \mathbf{u} \ d \mathbf{u}.
	\]
	By Theorem \ref{theo:count_integral}, 
	\[
	P(L,\mathbf{1} \cdot \gamma) := \int_{\mathcal{M}_{g,n}} s(X,\mathbf{1} \cdot \gamma, L) \ d\widehat{\mu}_{wp}(X) = \int_{\mathbf{1} \cdot \mathbf{u} \leq L} V_{g,n}(\gamma, \mathbf{u}) \cdot \mathbf{u} \ d \mathbf{u},
	\]
	Putting everything together we deduce
	\begin{equation}
	\label{eq:integral_bound}
	\int_{\mathcal{M}_{g,n}} s(X, \mathbf{a} \cdot \gamma,L) \ d\widehat{\mu}_{\text{wp}}(X) \leq \kappa(\gamma,\mathbf{a}) \cdot \prod_{i=1}^k \frac{1}{a_i^2} \cdot P(L,\mathbf{1} \cdot \gamma).
	\end{equation}
	$ $
	
	Let $X \in \mathcal{M}_{g,n}$ and $L > 0$ be arbitrary. We denote by $U_X(1) \subseteq \mathcal{M}_{g,n}$ the closed ball of radius $1$ centered at $X$ in the quotient symmetric Thurston metric. By definition, $Y\in U_X(1)$ if and only if there is a choice of markings for $X$ and $Y$ (allowing us to consider them as points in $\mathcal{T}_{g,n}$) such that for all $\lambda \in \mathcal{ML}_{g,n}$ the following bounds hold:
	\[
	e^{-1}\leq \frac{\ell_X(\lambda)}{\ell_Y(\lambda)} \leq e.
	\]
	 In particular, if $Y\in U_X(1)$ then
	 \[
	 s(X, \mathbf{a} \cdot \gamma,L) \leq s(Y, \mathbf{a} \cdot \gamma, e L).
	 \]
	This observation gives the following rough bound:
	\begin{align*}
	s(X, \mathbf{a} \cdot \gamma,L) \cdot \widehat{\mu}_{\text{wp}}(U_X(1)) &= \int_{\mathcal{M}_g} \mathbbm{1}_{U_X(1)}(Y) \cdot s(X,\mathbf{a} \cdot \gamma,L) \ d\widehat{\mu}_{wp}(Y)\\
	&\leq \int_{\mathcal{M}_g} \mathbbm{1}_{U_X(1)}(Y) \cdot s(Y, \mathbf{a} \cdot \gamma, eL) \ d\widehat{\mu}_{wp}(Y)\\
	&\leq \int_{\mathcal{M}_g} s(Y, \mathbf{a} \cdot \gamma, eL) \ d\widehat{\mu}_{wp}(Y)\\
	&\leq \kappa(\gamma,\mathbf{a}) \cdot \prod_{i=1}^k \frac{1}{a_i^2} \cdot P(eL,\mathbf{1} \cdot \gamma),
	\end{align*}
	where the last inequality follows from (\ref{eq:integral_bound}). Notice $\widehat{\mu}_{\text{wp}}(U_X(1)) > 0$ because $U_X(1)$ is a neighborhood of $X$ and $\widehat{\mu}_{\text{wp}}$ has full support on $\mathcal{M}_{g,n}$. We deduce
	\[
	\frac{s(X, \mathbf{a} \cdot \gamma, L)}{L^{2N}} \leq \kappa(\gamma, \mathbf{a})  \cdot \frac{e^{6g-6+2n}}{\widehat{\mu}_{\text{wp}}(U_X(1))} \cdot \frac{P(eL,\mathbf{1}\cdot \gamma)}{(eL)^{6g-6+2n}} \cdot \prod_{i=1}^k \frac{1}{a_i^2}.
	\]
	By Proposition \ref{prop:mir_freq}, 
	\[
	\lim_{L \to \infty} \frac{P(eL,\mathbf{1}\cdot \gamma)}{(eL)^{6g-6+2n}} = c(\mathbf{1} \cdot \gamma) > 0.
	\]
	Let $L_0 >0$ be big enough so that 
	\[
	\frac{P(eL,\mathbf{1}\cdot \gamma)}{(eL)^{6g-6+2n}} \leq 2 \cdot c(\mathbf{1} \cdot \gamma) 
	\]
	for all $L \geq L_0$. It follows that 
	\[
	s(X, \mathbf{a} \cdot \gamma, L) \leq \kappa(\gamma,\mathbf{a}) \cdot \frac{2 \cdot e^{6g-6+2n} \cdot c(\mathbf{1} \cdot \gamma)}{\widehat{\mu}_{\text{wp}}(U_X(1))}  \cdot \prod_{i=1}^k \frac{1}{a_i^2} 
	\]
	for all $L \geq L_0$. As $\kappa(\gamma, \mathbf{a})$ takes only finitely many values when $\mathbf{a}$ ranges over $\mathbf{N}^k$, see Theorem \ref{theo:count_integral}, this finishes the proof.
\end{proof}
$ $

We are now ready to prove Theorem \ref{theo:vol_freq}.\\

\begin{proof}[Proof of Theorem \ref{theo:vol_freq}]
	As in Remark \ref{rem:b(X,L)_no_bound}, for every $X \in \mathcal{M}_{g,n}$ and every $L > 0$ we consider the counting function
	\[
	b(X,L):= \# \{\alpha \in \mathcal{ML}_{g,n}(\mathbf{Z}) \ | \ \ell_X(\alpha) \leq L \}.
	\]
	By the definition of the Thurston measure, see the paragraph following (\ref{ML_counting_measure}), for every $X \in \mathcal{M}_{g,n}$ we have
	\[
	\lim_{L \to \infty} \frac{b(X,L)}{L^{6g-6+2n}} = B(X).
	\]
	It follows that we can write
	\[
	b_{g,n} := \int_{\mathcal{M}_{g,n}} B(X) \ d\widehat{\mu}_{wp}(X) = \int_{\mathcal{M}_{g,n}}  \lim_{L \to \infty} \frac{b(X,L)}{L^{6g-6+2n}} \  d\widehat{\mu}_{wp}(X) .
	\]
	Fix $X \in \mathcal{M}_{g,n}$. Notice that we can decompose $b(X,L)$ as the sum of the counting functions $s(X,\gamma,L)$ with $\gamma$ ranging over all mapping class group orbits of integral multi-curves on $S_{g,n}$:
	\[
	b(X,L) = \sum_{\gamma \in \mathcal{ML}_{g,n}(\mathbf{Z})/\text{Mod}_{g,n}} s(X,\gamma,L).
	\]
	Let $\mathcal{C}_{g,n}$ be the finite set of all topological types of unweighted multi-curves on $S_{g,n}$. For every $\gamma := \{\gamma_1,\dots,\gamma_k\} \in \mathcal{C}_{g,n}$ choose an arbitrary ordering of its components; we will also denote the corresponding ordered topological multi-curve by $\gamma := (\gamma_1,\dots,\gamma_k)$. We write
	\[
	b(X,L) = \sum_{\gamma \in \mathcal{C}_{g,n}} \sum_{\mathbf{a} \in \mathbf{N}^k} s(X,\mathbf{a} \cdot \gamma, L).
	\]
	As the outside sum in this equality is finite, we deduce
	\begin{equation}
	\label{eq:out_lim_ex}
	\lim_{L \to \infty} \frac{b(X,L)}{L^{6g-6+2n}} = \sum_{\gamma \in \mathcal{C}_{g,n}} \ \lim_{L \to \infty} \sum_{\mathbf{a} \in \mathbf{N}^k} \frac{s(X,\mathbf{a}\cdot \gamma,L)}{L^{6g-6+2n}}.
	\end{equation}
	We now exchange the limit in the right hand side of this equality with the infinite inside sum by using the dominated convergence theorem. By Theorems \ref{theo:mir_count} and \ref{theo:lead}, for every $\mathbf{a} \in \mathbf{N}^k$ we have
	\[
	\lim_{L \to \infty} \frac{s(X, \mathbf{a} \cdot \gamma,L)}{L^{6g-6+2n}}  = \frac{c(\mathbf{a} \cdot \gamma) \cdot B(X)}{b_{g,n}}.
	\]
	Lemma \ref{lem:count_bound} provides constants $C > 0$ and $L_0>0$ such that for all $\mathbf{a}:= (a_1,\dots,a_k) \in \mathbf{N}^k$ and all $L \geq L_0$,
	\[
	\frac{s(X, \mathbf{a} \cdot \gamma,L)}{L^{6g-6+2n}} \leq C \cdot \prod_{i=1}^k \frac{1}{a_i^2}.
	\]
	Notice that
	\[
	\sum_{\mathbf{a} \in \mathbf{N}^k} \prod_{i=1}^k \frac{1}{a_i^2} = \zeta(2)^k < +\infty,
	\]
	so the dominated convergence theorem applies. We deduce
	\[
	\lim_{L \to \infty} \sum_{\mathbf{a}\in \mathbf{N}^k} \frac{s(X,\mathbf{a} \cdot \gamma,L)}{L^{6g-6+2n}} = \sum_{\mathbf{a} \in \mathbf{N}^k} \frac{c(\mathbf{a} \cdot \gamma) \cdot B(X)}{b_{g,n}}
	\]
	for every $\gamma \in \mathcal{C}_{g,n}$. It follows from (\ref{eq:out_lim_ex}) that
	\[
	\lim_{L \to \infty} \frac{b(X,L)}{L^{6g-6+2n}} = \sum_{\gamma \in \mathcal{C}_{g,n}}  \sum_{\mathbf{a} \in \mathbf{N}^k} \ \frac{c(\mathbf{a} \cdot \gamma) \cdot B(X)}{b_{g,n}} 
	= \sum_{\gamma \in \mathcal{ML}_{g,n}(\mathbf{Z})/\text{Mod}_{g,n}} \frac{c(\gamma) \cdot B(X)}{b_{g,n}}.
	\]
	This equality holds for every $X \in \mathcal{M}_{g,n}$, so we have
	\[
	\int_{\mathcal{M}_{g,n}}  \lim_{L \to \infty} \frac{b(X,L)}{L^{6g-6+2n}} \ d\widehat{\mu}_{\text{wp}}(X) = \int_{\mathcal{M}_{g,n}} \sum_{\gamma \in \mathcal{ML}_{g,n}(\mathbf{Z})/\text{Mod}_{g,n}} \frac{c(\gamma) \cdot B(X)}{b_{g,n}} \ d\widehat{\mu}_{\text{wp}}(X).
	\]
	Fubini's theorem for non-negative functions gives
	\begin{align*}
	&\int_{\mathcal{M}_{g,n}} \sum_{\gamma \in \mathcal{ML}_{g,n}(\mathbf{Z})/\text{Mod}_{g,n}} \frac{c(\gamma) \cdot B(X)}{b_{g,n}} \ d\widehat{\mu}_{\text{wp}}(X) \\
	&= \sum_{\gamma \in \mathcal{ML}_{g,n}(\mathbf{Z})/\text{Mod}_{g,n}} \ \int_{\mathcal{M}_{g,n}} \frac{c(\gamma) \cdot B(X)}{b_{g,n}} \ d\widehat{\mu}_{\text{wp}}(X).
	\end{align*}
	By the definition of $b_{g,n}$,
	\[
	 \int_{\mathcal{M}_{g,n}}  \frac{c(\gamma) \cdot B(X)}{b_{g,n}} \ d\widehat{\mu}_{\text{wp}}(X) = c(\gamma).
	\]
	Putting everything together we conclude
	\[
	b_{g,n} = \sum_{\gamma \in \mathcal{ML}_{g,n}(\mathbf{Z})/\text{Mod}_{g,n}} c(\gamma),
	\]
	finishing the proof.
\end{proof}
$ $

Directly from Theorems \ref{theo:joint_freq_rel} and  \ref{theo:vol_freq} we obtain an analogous relation between the constant $a_{g,n}$ and the joint frequencies $c(\gamma_1,\gamma_2)$; this finishes the proof of Theorem \ref{theo:main_res_5}.\\

\begin{theorem}
	\label{theo:var_freq}
	For any integers $g,n \geq 0$ such that $2 -2g -n < 0$,
	\[
	a_{g,n} = \sum_{\gamma_1,\gamma_2 \in \mathcal{ML}_{g,n}(\mathbf{Z})/\text{Mod}_{g,n}} c(\gamma_1,\gamma_2).
	\]
\end{theorem}
$ $

\begin{proof}
	By Theorem \ref{theo:joint_freq_rel} we have
	\[
	c(\gamma_1,\gamma_2) = \frac{a_{g,n}}{b_{g,n}^2} \cdot c(\gamma_1) \cdot c(\gamma_2)
	\]
	for every $\gamma_1,\gamma_2 \in \mathcal{ML}_{g,n}(\mathbf{Z})$. Theorem \ref{theo:vol_freq} shows
	\[
	b_{g,n} = \sum_{\gamma \in \mathcal{ML}_{g,n}(\mathbf{Z})/\text{Mod}_{g,n}} c(\gamma).
	\]
	It follows that
	\begin{align*}
	&\sum_{\gamma_1,\gamma_2 \in \mathcal{ML}_{g,n}(\mathbf{Z})/\text{Mod}_{g,n}} c(\gamma_1,\gamma_2) \\
	&= \sum_{\gamma_1,\gamma_2 \in \mathcal{ML}_{g,n}(\mathbf{Z})/\text{Mod}_{g,n}} \frac{a_{g,n}}{b_{g,n}^2} \cdot c(\gamma_1) \cdot c(\gamma_2)\\
	&= \frac{a_{g,n}}{b_{g,n}^2} \cdot \left( \sum_{\gamma_1\in \mathcal{ML}_{g,n}(\mathbf{Z})/\text{Mod}_{g,n}} c(\gamma_1) \right) \cdot \left( \sum_{\gamma_2\in \mathcal{ML}_{g,n}(\mathbf{Z})/\text{Mod}_{g,n}} c(\gamma_2) \right)\\
	&=  \frac{a_{g,n}}{b_{g,n}^2} \cdot b_{g,n} \cdot b_{g,n}\\
	&= a_{g,n},
	\end{align*}
	finishing the proof.
\end{proof}
$ $

\section{Open Questions}

$ $

\textit{Computing $a_{g,n}$ and joint frequencies.} In Theorem 5.3 of \cite{Mir08b}, Mirzakhani gives formulas for the frequencies $c(\gamma)$ and the constant $b_{g,n}$ in terms of leading coefficients of Weil-Petersson volume polynomials of moduli spaces of complete, finite volume hyperbolic surfaces with geodesic boundary. As such polynomials can be computed recursively, see \S 5 in \cite{Mir07a}, this provides an algorithmic procedure for computing the frequencies $c(\gamma)$ and the constant $b_{g,n}$.\\

\begin{question}
	\label{q:1}
	For any pair of integers $g,n \geq 0$ such that $2-2g-n < 0$, provide an algorithmic procedure for computing
	\[
	a_{g,n} := \int_{\mathcal{M}_{g,n}} B(X)^2 \ d \widehat{\mu}_{wp}(X).
	\]
\end{question}
$ $

\begin{question}
	\label{q:2}
	For any pair of integral multi-curves $\gamma_1,\gamma_2 \in \mathcal{ML}_{g,n}(\mathbf{Z})$, provide an algorithmic procedure for computing
	\[
	c(\gamma_1,\gamma_2) := \lim_{L \to \infty} \frac{1}{L^{12g-12+4n}} \int_{\mathcal{M}_{g,n}} s(X,\gamma_1,L) \cdot s(X,\gamma_2,L) \ d\widehat{\mu}_{\text{wp}}(X).
	\]
\end{question}
$ $

Notice that by Theorems \ref{theo:main_res_4} and  \ref{theo:main_res_5} and the work of Mirzakhani cited above, Questions \ref{q:1} and \ref{q:2} are essentially equivalent.\\

\textit{Relating $a_{g,n}$ to moduli spaces of quadratic differentials.} Recall that $b_{g,n}$, the integral of $B$ with respect to the Weil-Petersson measure $\widehat{\mu}_{\text{wp}}$ on $\mathcal{M}_{g,n}$, corresponds to the Masur-Veech measure of the principal stratum of $Q\mathcal{M}_{g,n}$, the moduli space of connected, integrable, meromorphic quadratic differentials of genus $g$ with $n$ marked points. \\

\begin{question}
	Is there a meaningful interpretation of the integral
	\[
	a_{g,n} := \int_{\mathcal{M}_{g,n}} B(X)^2 \ d \widehat{\mu}_{wp}(X)
	\]
	in terms of the moduli space $Q\mathcal{M}_{g,n}$?
\end{question}
$ $

\textit{Large genus aymptotics.} For every pair of integers $g,n \geq 0$ satisfying $2 -2g - n <0$, consider the probability space $(\mathcal{M}_{g,n}, \widehat{\mu}_{\text{wp}}/m_{g,n})$, where $m_{g,n} := \widehat{\mu}_{\text{wp}}(\mathcal{M}_{g,n})$; each moduli space $\mathcal{M}_{g,n}$ has a different Weil-Petersson measure but we denote them all by $\widehat{\mu}_{\text{wp}}$. Each one of these moduli spaces carries a Mirzakhani function $B_{g,n} \colon \mathcal{M}_{g,n} \to \mathbf{R}_{>0}$. For a fixed $n \geq 0$ we are interested in the behavior of $B_{g,n}$ as $g \to \infty$.\\

\begin{question}
	\label{q:4}
	What are the asymptotics of $\mathbf{Var}(B_{g,n}(X))$ as $g \to \infty$? 
\end{question}
$ $

Answering Question \ref{q:4} could provide meaningful insight on the behavior in the large genus regime of the dependency with respect to the hyperbolic structure of the leading coefficient of the asymptotics of counting problems for simple closed geodesics. Indeed, by Theorems \ref{theo:mir_count} and \ref{theo:lead}, such dependency is precisely given by $B_{g,n}$, and Chebyshev's inequality shows that for every $a \geq 0$,
\[
\mathbf{P}(|B_{g,n}(X)- \mathbf{E}(B_{g,n}(X))| \geq a) \leq \frac{\mathbf{Var}(B_{g,n}(X))}{a^2}.
\]
$ $

Inspired by Mirzakhani's work on the geometry of random hyperbolic surfaces of large genus sampled according to the the probability measures $\widehat{\mu}_{\text{wp}}/m_{g,n}$, see \cite{Mir13}, it would be very interesting to know more about the geometry of random hyperbolic surfaces of large genus sampled according to the probability measures $B_{g,n}(X) 
\ \allowbreak d \widehat{\mu}_{\text{wp}}(X) / b_{g,n}$. In particular, the following question seems especially interesting.\\

\begin{question}
	For $\epsilon > 0$ small enough, what are the asymptotics as $g \to \infty$ of the probability that a random hyperbolic surface sampled according to $B_{g,n}(X) \  \allowbreak d \widehat{\mu}_{\text{wp}}(X) / b_{g,n}$ exhibits a simple closed geodesic of length $\leq \epsilon$?
\end{question}
$ $


\bibliographystyle{amsalpha}


\bibliography{bibliography}

\end{document}